\documentclass[12pt]{amsart}
\usepackage{amssymb,latexsym}
\usepackage[mathscr]{eucal}
\usepackage{graphics}
\usepackage{ulem}

\theoremstyle{plain}
\newtheorem{theorem}{Theorem}[section]
\newtheorem{corollary}[theorem]{Corollary}
\newtheorem{lemma}[theorem]{Lemma}
\newtheorem{proposition}[theorem]{Proposition}
\newtheorem{remark}[theorem]{Remark}
\newtheorem{definition}[theorem]{Definition}
\newtheorem{example}[theorem]{Example}

\def \co*{{\text{c}_{\text{o}}^*}}
\def \D*{$\Delta_{1/2}$-condition}

\begin{document}

\title{B(H) Lattices, Density and Arithmetic Mean Ideals}
\author{Victor Kaftal}
\address{University of Cincinnati\\
          Department of Mathematics\\
          Cincinnati, OH, 45221-0025\\
          USA}
\email{victor.kaftal@uc.edu}
\author{Gary Weiss}\thanks{The first named author was partially
supported by The Charles Phelps Taft Research Center,
the second named author was partially supported by NSF Grants DMS 95-03062
      and DMS 97-06911 and by The Charles Phelps Taft Research Center.}
\email{gary.weiss@uc.edu}

\keywords{soft ideals, operator ideals, arithmetic means}
\subjclass{Primary: 47B47, 47B10; Secondary: 46A45, 46B45, 47L20}

\begin{abstract}
This paper studies lattice properties of operator ideals in B(H) and their applications to the arithmetic mean ideals which were introduced in \cite {DFWW} and further studied in the project \cite {vKgW02}-\cite{vKgW04-2nd order and cancellation} of which this paper is a part.
It is proved that the lattices of all principal ideals, of principal ideals with a generator that satisfies the $\Delta_{1/2}$-condition, of arithmetic mean stable principal ideals (i.e., those with an am-regular generator), and of arithmetic mean at infinity stable principal ideals (i.e., those with an am-$\infty$ regular generator) are all both upper and lower dense in the lattice of general ideals. This means that between any ideal and an ideal
(nested above or below respectively) in one of these sublattices, lies another ideal in that sublattice.
Among the applications: a principal ideal is am-stable (and similarly for am-$\infty$ stable principal ideals) if and only if any (or equivalently, all) of its first order arithmetic mean ideals are am-stable, such as its am-interior, am-closure and others.
A principal ideal $I$ is am-stable (and similarly for am-$\infty$ stable principal ideals) if and only if it satisfies any (equivalently, all) of  the first order equality cancellation properties, e.g, $J_a = I_a \Rightarrow J= I$. These cancellation properties can fail even for am-stable countably generated ideals.
It is proven that while the inclusion cancellation $J_a \supset I_a$ implies $J \supset I$ does not hold in general, even for $I$ am-stable and principal, there is always a largest ideal $\widehat I$ for which $J_a \supset I_a  \Rightarrow J \supset \widehat I$.
Furthermore, if $I = (\xi)$ is principal, then $\widehat I$ is principal as well
and $\widehat I = (\widehat \xi)$ for a sequence $\widehat \xi$ optimal in the majorization sense, that is, $\eta \ge \widehat \xi$ asymptotically for all monotone nonincreasing $\eta \in c_o$ for which $\sum_{1}^{n}\eta_j \ge \sum_{1}^{n}\xi_j$ for every $n$ (i.e., with arithmetic means $\eta_a \ge \xi_a$).
In particular, $\widehat {\omega^{1/p}} =  \omega^{1/p'}$ for the harmonic sequence $\omega$, $0<p<1$, and $1/p - 1/p' = 1$.
\end{abstract}
\maketitle

\section{\leftline{Introduction}}\label{S:1}

Operators ideals, the two-sided ideals of $B(H)$, have played an important role in operator theory and operator algebras since they were first studied by J. Calkin \cite{jC41} in 1941.

Many of the questions in the theory have involved the structure of the lattice $\mathscr L$ of all operator ideals and of its distinguished sublattices, like the lattice $\mathscr PL$ of all principal ideals. For instance, one basic question due to Brown, Pearcy and Salinas \cite{BPS71} was whether or not the ideal of compact operators $K(H)$ was the sum of two proper ideals. This was settled affirmatively in \cite{aBgW78} for any proper ideal strictly larger than $F$ (the ideal of finite rank operators) using the continuum hypothesis and the techniques employed led to the set-theoretic concept of groupwise density that has proved useful in point-set topology and abelian group theory (\cite{aB89}, \cite {aB90}, \cite {hM01}). Other work by Salinas (e.g., \cite{nS74}) studied increasing or decreasing nests of special classes of symmetric norm ideals (Banach ideals).

Another central topic in the theory was the study of commutator spaces (also called commutator ideals) on which much work was done from the early years. More recently, the introduction of cyclic cohomology in the 1980's by A. Connes (e.g., see \cite {aC94}) and the connection  with algebraic K-theory by M. Wodzicki in the 1990's (e.g., see \cite {mW94}) provided a powerful motivation for further work on operator ideals and in particular on commutator  spaces.
Arithmetic means of monotone sequences were first connected, albeit only implicitly, to commutator spaces in \cite{gW75} 
(see also \cite{gW80}-\cite{gW86} and survey \cite{gW05}) and then explicitly for the trace class in \cite {nK89}. They provided the key tool for the  full characterization of commutator  spaces achieved in \cite{DFWW} in terms of arithmetic mean operations on ideals. This lead to the definition of a number of  arithmetic mean ideals derived from an ideal $I$: the arithmetic mean $I_a$, the pre-arithmetic mean $_aI$, the am-interior $I^o$ and the am-closure $I^-$.  After \cite{DFWW}, the connection between arithmetic
means and operator ideals were further studied in \cite{nK98}, \cite{kDnK98}, \cite{kDnK05}  and in an ongoing program by the authors of this   paper which was announced in \cite{vKgW02} and which includes \cite{vKgW04-Traces}-\cite{vKgW04-2nd order and cancellation}.

 At the beginning of this project it soon became apparent that to investigate questions such as  ``how many traces can an ideal support'', i.e., the codimension of the commutator space of the ideal, which question formed the main focus of  \cite {vKgW04-Traces},
  a more systematic  study of the actions of the arithmetic mean operations on operator ideals was necessary and the action of   the arithmetic mean at infinity needed also to be introduced formally (see Section 2 herein for definitions). This program was carried out in part in \cite {vKgW04-Traces}  and further in \cite {vKgW04-Soft}, but many questions, and in particular those involving arithmetic mean cancellation properties, led to an analysis of the relation between operator lattices and arithmetic mean operations. The goal of this paper is to provide this analysis and its applications.

Section 2 introduces the notations and some of the necessary
preliminaries for arithmetic mean ideals.

In Section 3 we study the density properties of ideal sublattices of $\mathscr L$.
Given two lattices $\mathscr L_1 \subset \mathscr L_2$, $\mathscr L_1$ is said to be upper dense in $\mathscr L_2$ if for every pair of ideals $I_1\subsetneq I_2$ with $I_i \in \mathscr L_i$, there is an intermediate ideal $I_1\subsetneq L \subsetneq I_2$ with $L \in \mathscr L_1$.
The notion of lower density is similar. We consider the following sublattices of $\mathscr L$.
(See the next section for their precise definitions.)\\
$\bullet~ \mathscr PL$: principal ideals,\\
$\bullet~\mathscr {L}_{\aleph_o}$: countably generated ideals,\\
$\bullet~\Delta_{1/2}\mathscr {PL}$: principal ideals with a generator
that satisfies the $\Delta_{1/2}$-condition, \\
$\bullet~\mathscr {SPL}$: principal ideals with an am-regular generator, \\
$\bullet~\mathscr {S_\infty PL}$ principal ideals with an am-$\infty$
regular generator.

A first result, which is then used throughout the paper, is that
$\mathscr PL$ and $\mathscr {L}_{\aleph_o}$ are upper dense in $\mathscr {L}$ (Corollary \ref {C: count gen upper density}).
These lattices are, however, not lower dense in $\mathscr {L}$ as every
nonzero principal ideal $I$ has a gap undeneath it, i.e., an ideal
$M\subsetneq I$ with no other ideals in-between (Corollary \ref{C: gaps princ}).

The main technical results of this section and of the paper are that
$\Delta_{1/2}\mathscr {PL}$,
  $\mathscr {SPL}$, and $\mathscr {S_\infty PL}$ are both upper and
lower dense in $\mathscr PL$
(Theorems  \ref{T: Delta1/2-PL strong density and strong gaps},
\ref{T: SPL density in PL} and \ref{T: SinftyPL density}).
The proofs require us
to construct between a ``nice'' sequence (the
s-numbers of the generator)
  and a possibly ``bad'' but  comparable sequence, a new ``nice''
intermediate sequence that is inequivalent (in
  the ideal sense) to either of the two given sequences. The key tool
in the case of $\mathscr {SPL}$ and
  $\mathscr {S_\infty PL}$ are the Potter-type conditions satisfied by
the regular or $\infty$-regular sequences.
  For the am-case, these are just the properties of the Matuszewska
$\beta$ -index studied in \cite [Theorem 3.10]{DFWW}. For
  the am-$\infty $ case, these are the properties of the Matuszewska
$\alpha$-index and were introduced in
\cite [Theorem 4.12]{vKgW04-Traces}.

In Section 4 we show that a Banach ideal is am-stable if and only if it is the union
of am-stable principal ideals (Proposition \ref{P:complete}). While it is clear that every ideal is the union of the principal ideals that it contains, it is not known which ideals are the union of an increasing chain of principal ideals. Assuming the continuum hypothesis, we show that every ideal is the union of an increasing chain of countably generated ideals, and we prove that if the ideal is the power of a Banach ideal then it is also the union of an increasing chain of principal ideals (see Proposition \ref{P: nested unions and the continuum hypothesis}).

In Section 5 we use the density results from Section 3 to analyze the relation between the properties of an ideal and those of its first order arithmetic mean ideals (resp.,  arithmetic mean ideals at infinity, see the next section for definitions).
The first results (Theorems \ref {T: I am-stable iff} and \ref {T: I am-inf stable iff})
are: for principal  ideals, the am-stability (resp., am-$\infty$ stability) of $I$ is equivalent to
  the am-stability of any of its first order arithmetic mean ideals (resp., first order arithmetic mean ideals at infinity). What is perhaps interesting is that this ``rigidity'' property does not extend even to countably generated ideals: Examples \ref{E:N} and \ref {E:L} exhibit ideals   that are not am-stable but have various first-order am-ideals that are am-stable.

In Section 6 we apply density to investigate first order cancellation properties for arithmetic mean operations.
(Second order cancellation properties are more complex and are studied in \cite{vKgW02} and \cite{vKgW04-2nd order and cancellation}.)
Equality cancellation questions ask: Under which conditions on an ideal $I$, does $J_a=I_a$ imply $J=I$? (and similarly for the other derived am and am-$\infty$ ideals).
Inclusion cancellation questions ask: Under which conditions on an ideal $I$ does $I_a \subset J_a$ imply $I\subset J$, or $J_a\subset I_a$ imply $J\subset I$?
(and similarly for the other derived am and am-$\infty$ ideals).
Beyond the elementary cancellations
(e.g., for fixed $I$, $J_a\subset I_a  \Rightarrow J\subset I$ if and only if $I$ is am-closed (see Lemma \ref{L:closed open})), we know of no natural conditions for general am-cancellation to hold.  Examples \ref{E:N} and \ref{E:L} show that the equality am-cancellations can fail even for am-stable countably generated ideals.

For principal ideals however, the necessary and sufficient condition for equality cancellations is that the principal ideal $I$ is am-stable 
(resp., am-$\infty$ stable) (see Theorem \ref{T: cancellation}).
One key ingredient in the proof is the fact that for countably generated ideals,
and hence for principal ideals, am-stability is equivalent to am-closure,
and similarly for the am-${\infty}$ case (see \cite[Theorems 2.11 and 3.5]{vKgW04-Soft}).
For other cancellations we need instead a technical result that shows that if a principal ideal is am-open but not am-stable, then it properly contains and is properly contained in two other principal ideals that share with it, respectively, the same am-open envelope and the same am-interior, and similarly for the am-$\infty$ case (Proposition \ref{P:open'}).

For the nontrivial inclusion cancellations, am-stability is not sufficient even in the principal ideal case. For instance, for every principal ideal $I$, there is a principal ideal $J$ with $_aJ\subset\,_aI$ but $J\not\subset I$ (Proposition \ref{P:open} ), i.e., no cancellation whatsoever is possible. \linebreak

For the cancellation $J_a \supset I_a \Rightarrow J \supset I$, a stronger condition than am-stability is required for the ideal $I$. Indeed, we show that for every ideal $I$ there is an ``optimal'' ideal $\widehat I$ for which $J_a \supset I_a  \Rightarrow J \supset\widehat I$.  If   $I=(\xi)$, then $\widehat{(\xi)}= (\widehat{\xi})$. Furthermore, the sequence $\widehat{\xi}$ is also optimal in the majorization sense, i.e., if  $\eta_a \ge \xi_a$ then $\eta \ge \widehat{\xi}$ asymptotically (Theorem \ref {T: 4}).
Sequences for which $(\xi)= (\widehat{\xi})$ are necessarily regular (Proposition \ref{P:Gg is stable}) and they form a distinguished proper subclass of the regular sequences.
Indeed, for the harmonic sequence $\omega$, $\widehat{\omega^p} \asymp \omega^{p'}$, where $0 < p < 1$ and $\frac{1}{p}-\frac{1}{p'} = 1$  (Corollary \ref{C: 6}). Thus $\omega^{1/2}$ is regular but $\omega^{1/2}\not\asymp \widehat{\omega^{1/2}} \asymp \omega$.
This result is linked to the single commutator problem implicit in \cite [Chapter 7]{DFWW}
on whether containment in the class of single $(I,B(H))$-commutators, $[I,B(H)]_1$,
of a finite rank nonzero trace operator implies $\omega^{1/2} \in \Sigma(I)$.

\section{\leftline{Notations and Preliminary results}}\label{S:2}

Calkin \cite{jC41} established the inclusion preserving lattice isomorphism $I \rightarrow \Sigma(I) := \{s(X) \mid X \in I\}$ between two-sided ideals $I$ of $B(H)$, the algebra of bounded linear operators on a separable   infinite-dimensional complex Hilbert space $H$ and the characteristic sets $\Sigma \subset \co*$, i.e., the hereditary (solid) ampliation invariant (see below) subcones of the collection $\co*$ of sequences decreasing to $0$. Here $s(X)$ denotes the sequence of the s-numbers of the compact operator $X$ and the inverse of this correspondence maps a characteristic set $\Sigma$ to the ideal generated by the collection of diagonal operators $\{diag\,\xi \mid \xi \in \Sigma\}$.

Two sequence operations, the arithmetic mean acting on $\co*$ and the
arithmetic mean at infinity acting on $(\ell^1)^*$ (the monotone noncreasing summable
sequences)  respectively,
\[\xi_a := \left<\frac{1}{n}\sum_1^n \xi_j\right>
  \quad \text{and} \quad \xi_{a_\infty} :=
\left<\frac{1}{n}\sum_{n+1}^\infty \xi_j\right> \]
are essential for the study of commutator spaces (i.e., commutator ideals) and hence traces on ideals, as mentioned in the introduction (e.g., see \cite{DFWW} and \cite{vKgW02}-\cite{vKgW04-Soft}).
For the readers' convenience we list the definitions and first properties from \cite[Sections 2.8 and 4.3]{DFWW} of the ideals derived via arithmetic mean operations (am-ideals for short).

If $I$ is an ideal, then the arithmetic mean ideals $_aI$ and $I_a$,
called respectively the
  \textit{pre-arithmetic mean} and \textit{arithmetic mean} of $I$,
are the ideals with characteristic sets
\[
\Sigma(_aI) := \{\xi \in \co* \mid \xi_a \in \Sigma(I)\}
\]
\[
\Sigma(I_a) := \{\xi \in \co* \mid \xi = O(\eta_a)~\text{for some}~
\eta \in \Sigma(I)\}.
\]
The  \textit{arithmetic mean-closure} $I^-$ and \textit{arithmetic
mean-interior} $I^o$ of an ideal
(am-closure and am-interior for short) are defined as
\[
I^- :=\, _a(I_a) \qquad \text{and} \qquad I^o := (_aI)_a.
\]
For any ideal $I$, the following \textit{5-chain of inclusions} holds:
\[
_aI \subset I^o \subset I \subset I^- \subset I_a
\]
and the identites
\[
I_a=(_a(I_a))_a \qquad \text{and} \qquad _aI= \,_a((_aI)_a).
\]
Other first order arithmetic mean ideals derived from a given ideal
$I$ and introduced in \cite{vKgW02}, \cite{vKgW04-Soft} are the
largest am-closed ideal $I_-$ contained in $I$ and the smallest
am-open ideal $I^{oo}$ containing $I$.

We take this opportunity the make a shift in the terminology introduced in \cite{vKgW02}.
There we called 2nd order arithmetic mean ideals those obtained from applying twice the basic am and pre-am operations to an ideal,
e.g., the am-closure $_a(I_a)$ and the am-interior $(_aI)_a$.
In later work we found a commonality of properties between the ideals $I_a$, $_aI$, $_a(I_a)$, $(_aI)_a$ and the ideals $I_-$ and $I^{oo}$ mentioned above
that motivated us to now call them all first order am-ideals, while ideals of the kind $I_{a2}$, $_{a2}(I_{a2})$, and so on will be called 2nd order am-ideals.

As a consequence of one of the main results in \cite{DFWW}, an ideal is am-stable (i.e., $I=\,_aI$, or equivalently, $I=I_a$)
if and only if $I=[I,B(H)]$ if and only if $I$ supports no nonzero trace.
In particular, ideals contained in the trace class $\mathscr L_1$ cannot be am-stable.
For them, the arithmetic mean at infinity is the relevant operation.
It was often used as a sequence operation in connection with operator ideals (see for instance \cite{hM01}, \cite{nK87}, \cite{DFWW}, \cite {mW02})
and its action on operator ideals was studied in \cite[Section 4]{vKgW04-Traces}.
For the readers' convenience we list the definitions and first properties.

If $I$ is an ideal, then the arithmetic mean at infinity ideals
$_{a_\infty}I$ and $I_{a_\infty}$
  are the ideals with characteristic sets
\[
\Sigma(_{a_\infty}I) := \{\xi \in (\ell^1)^* \mid \xi_{a_\infty} \in
\Sigma(I)\}
\]
\[
\Sigma(I_{a_\infty}) := \{\xi \in \text{c}_{\text{o}}^{*} \mid \xi =
  O(\eta_{a_\infty})  ~\text{for some}~ \eta \in \Sigma(I \cap \mathscr L_1)\},
\]
and the other derived  am-$\infty$ ideals are the am-$\infty$ closure
$I^{-\infty} :=\, _{a_\infty}(I_{a_\infty})$,
the  am-$\infty$ interior $I^{o\infty} := (_{a_\infty}I)_{a_\infty}$,
the largest am-$\infty$ closed ideal,
$I_{-\infty}$, contained in $I$ and the smallest am-$\infty$ open ideal, $I^{oo\infty}$, 
containing $I\cap se(\omega)$ where
$\omega = <\frac{1}{n}>$ is the harmonic sequence and $se(\omega)$ is
the ideal whose characteristic set consists of sequences
$o(\omega)$.

The am-$\infty$ analog of the 5-chain of inclusions are
\[
_{a_\infty}I \subset I^{o\infty} \subset I \cap se(\omega) \qquad \text{and} \qquad
I  \cap \mathscr L_1 \subset I^{-\infty}\subset  I_{a_\infty} \cap \mathscr L_1
\]
and the following identites hold
\[
I_{a_{\infty}}=(_{a_{\infty}}(I_{a_{\infty}}))_{a_{\infty}} \qquad \text{and} \qquad _{a_{\infty}}I= \,_{a_{\infty}}((_{a_{\infty}}I)_{a_{\infty}}).
\]
An ideal $I$ is am-$\infty$ stable ($I=\,_{a_\infty}I$ or,
equivalently, $ I\subset \mathscr{L}_1$ and $I=I_{a_\infty}$)
if and only if $I = F +[I, B(H)]$ if and only if $I$ supports a
nonzero trace unique up to scalar multiples
(see \cite[Theorem 6.6]{vKgW04-Traces}).

For every ideal $I$, the lower and upper stabilizers
\[ st_a(I) := \bigcap_{m=0}^\infty \,_{a^m}I \subset I  \subset st^a(I) := \bigcup_{m=0}^\infty I_{a^m} \]
are, respectively, the largest am-stable ideal contained in ideal $I$ (possibly $\{0\}$)
and the smallest am-stable ideal containing $I$.
In particular, $st^a(F) = st^a(\mathscr L_1) = st^a((\omega))$ is the smallest am-stable ideal ($K(H)$ is the largest).
For the am-$\infty$ case and $I \ne \{0\}$,
\[\{0\} \ne st_{a_\infty}(I) := \bigcap_{m=0}^\infty \,_{a_\infty^m}I \subset I\]
and if $I \subset \mathscr L_1$, then
\[I \subset \bigcup_{m=0}^\infty I_{a_\infty^m}\]
are the largest am-$\infty$ stable ideal contained in $I$
and the smallest am-$\infty$ stable ideal containing $I$, respectively.
In particular, $F$ is the smallest am-$\infty$ stable ideal and the largest am-$\infty$ stable ideal is $st_{a_\infty}(\mathscr L_1) = st_{a_\infty}((\omega))$ with characteristic set \[\{\xi \in \co* \mid \sum \xi_n\,log^pn < \infty ~\text{for all}~ p>0\}.\]
(See \cite[Definition 4.16-Proposition 4.18 including proof]{vKgW04-Traces}.)

Of particular interests-and the main focus of this paper-are principal ideals.
If $X\in B(H)$ is the generator of the ideal $I$ and $\xi = s(X)$ is the s-number sequence of $X$,
then the diagonal operator $diag \,\xi$ is also a generator of $I$,
thus we will denote $I := (\xi)$, and by abuse of language,
we will simply say that a sequence $\xi\in \co*$ generates the ideal $(\xi)$.
The characteristic set $\Sigma((\xi))$ of a principal ideal $(\xi)$ consists of the sequences $\eta\in\co*$
for which $\eta = O(D_m(\xi))$ for some $m\in \mathbb{N}$, where the $m$-fold ampliation is
\[
\text{c}_{\text{o}}^* \owns \xi \rightarrow
D_m\xi:=\,<\xi_1,\dots,\xi_1,\xi_2,\dots,\xi_2,\xi_3,\dots,\xi_3,\dots>
\]
with each entry $\xi_i$ of $\xi$ repeated $m$-times.
More generally, $(D_t\xi)_n:=\xi\lceil \frac{n}{t}\rceil$ for all $t>0$, where $\lceil \frac{n}{t}\rceil$, denotes the smallest integer majorizing $\frac{n}{t}$.
A sequence is said to satisfy the \D* if $D_2\xi = O(\xi)$ (equivalently, if $\xi \asymp D_m\xi$  for all $m\in \mathbb N$).
In particular, if at least one of the generators satisfies the \D* condition, then $(\xi) \subset (\eta)$ if and only if $\xi = O(\eta)$.
Now, a principal ideal $(\xi)$ is am-stable if and only if the sequence $\xi$ is regular
(i.e., $\xi_a = O(\xi)$) since $(\xi)_a = (\xi_a)$ and every arithmetic mean sequence satifies the \D*.
Less trivially  since  $\xi_{a_\infty}$ in general does not satisfy the \D*, it is also true  that $(\xi)$ is am-$\infty$ stable if and only if
$\xi_{a_\infty} = O(D_m\xi)$ for some $m \in \mathbb N$ if and only if the sequence $\xi$ is $\infty$-regular
(i.e., $\xi_{a_\infty} = O(\xi)$) (see \cite[Theorem 4.12]{vKgW04-Traces}).

Operator ideals form a lattice $\mathscr{L}$ with inclusion as partial order and intersection (resp., sum) \linebreak
as meet (resp., join).
To avoid tedious disclaimers herein, we assume that $\mathscr{L}$ does not include the zero ideal $\{0\}$.
Since for any $\xi,\eta \in \co*$, $(\xi) \cap (\eta) = (\min(\xi,\,\eta) )$ and $(\xi) + (\eta) = (\xi+\eta )$,
the collection of all principal ideals forms a sublattice which we denote by $\mathscr{PL}$.
Since the intersections and sums of two am-stable (resp., am-$\infty$ stable) ideals is easily seen also to be am-stable (resp., am-$\infty$ stable),
the collection of all am-stable (resp., am-$\infty$ stable) principal nonzero ideals forms a lattice that we denote by $\mathscr{SPL}$
(resp., $\mathscr{S_{\infty}PL}$). Similarly, the collection of all nonzero principal ideals having a generator sequence that satisfies the \D* forms a lattice denoted by $\Delta_{1/2}\mathscr{PL}$.
Finally we will consider the lattice $\mathscr{L}_{\aleph_o}$ of
countably generated  nonzero ideals. Here too, we will say
by abuse of language that an ideal is generated by a countable
collection of sequences in $\co*$, meaning that it
is generated by the corresponding diagonal operators. Notice that if
$I$ is generated by the sequences $\eta^{(k)}$, then
$\eta\in\Sigma(I)$ if and only if $\eta = O(D_m(\eta^{(1)}+
\eta^{(2)}+ ... + \eta^{(k)})$ for some $m, k \in \mathbb N$.

\section{\leftline {The general ideal lattice, sublattices,
density and gaps}}\label{S:3}

\begin{definition}\label{L: density, strong density and gaps}
\item[(i)] A gap in a lattice of ideals is a nested pair of ideals in the lattice,
$I \subsetneq J$,  between which there is no ideal in
the lattice.
An upper (resp., lower) gap for an ideal $I$ is a gap
$I \subsetneq J$ (resp., $J \subsetneq I$).
\item[(ii)] For two nested lattices $\mathscr L' \subset \mathscr
L''$, $\mathscr L'$ is upper dense (resp.,
lower dense) in $\mathscr L''$ provided that
between every pair of ideals $I \in \mathscr L'$ and  $J\in \mathscr
L''$ with $I\subsetneq J$ (resp., $J\subsetneq I$) lies another ideal in the smaller lattice $\mathscr L'$.
\end{definition}

Although by assumption all ideals in a lattice are nonzero, we still say that $\{0\} \subset I$ is a gap in the lattice if there is no ideal $L$ in the lattice with $\{0\} \subsetneq L  \subsetneq I$, e.g., $F$ has a lower gap in $\mathscr {L}$.\\

First we discuss gaps in $\mathscr {L}$.

\begin{lemma}\label{L: gaps}
\item[(i)] A pair of ideals $I \subsetneq J$ is a gap in $\mathscr
{L}$ if and only if $J=I+(\xi)$
for some principal ideal $(\xi)$ and $I$ is maximal among the ideals $I\subset
L \subset J$ such that
$(\xi) \not\subset L$.
\item[(ii)] An ideal $J$ has a lower gap in $\mathscr {L}$ if and
only if $J$ can be
  decomposed into $J=N+(\xi)$ with $N$ a (possibly zero) ideal and
$(\xi)\not\subset N$.
\item[(iii)] An ideal $I$ has an upper gap in $\mathscr {L}$ if and
only if there is a principal ideal
$(\xi) \not\subset I$ such that if $I \subset I+(\eta)\subset
I+(\xi)$ for some principal ideal $(\eta)$, then
either  $I = I+(\eta)$ or $I+(\eta) = I+(\xi)$.
\end{lemma}

\begin{proof}
\item[(i)] If $J=I+(\xi)$ for some principal ideal $(\xi)$ and
$I\subset L \subset J$, then
$L \neq J$ if and only if $(\xi) \not\subset L$, thus $I \subset
J=I+(\xi)$ is a gap in $\mathscr {L}$
precisely when $I$ is maximal among the ideals $I\subset L \subset J$ such that
$(\xi) \not\subset L$. On the other hand, if $I \subset J$ is a gap
in $\mathscr {L}$, then for every
$\xi \in \Sigma (J)\setminus \Sigma (I) $ it follows that  $J=I+(\xi)$.
\item[(ii)] The condition is necessary by (i).
If $J=N+(\xi)$ with $N$ a (possibly zero) ideal and $(\xi)\not\subset N$, then
the collection $\mathscr C$  of all the ideals $L$ with $ N\subset L \subset
J$ and $L$ does not contain $(\xi)$ includes at least $N$.
It is immediate to see that the union of every chain in $\mathscr C$
is also an element of
$\mathscr C$. Let $M$ be the union of a maximal chain in $\mathscr
C$, which must exist by
  the Hausdorff Maximality Principle. Then by part (i), $M \subsetneq
J$ is a gap.
\item[(iii)] By (i) the condition is necessary. Indeed, if $I$ has an
upper gap in $\mathscr {L}$, the gap is of the form $I\subsetneq
I+(\xi)$ for some
principal ideal $(\xi)$. But then, any intermediate ideal $I +(\eta)$
must coincide with either $I$ or $I+(\xi)$.
The condition is also sufficient. Indeed, if $I\subset L \subset
I+(\xi)$ for some ideal $L$, then if $I \not = L$ then there is
a principal ideal $(\eta) \subset L\setminus I$ and hence $I
\subsetneq I+(\eta) \subset L \subset I+(\xi)$ and from
$I+(\eta) = I+(\xi)$ it follows that $L=I+(\xi)$, that is, $I
\subsetneq I+(\xi)$ forms a gap in $\mathscr L$.
\end{proof}

\noindent As a consequence, between any two nested ideals $I \subsetneq J$
  there is always a gap. Indeed, if $\xi \in \Sigma(J)\setminus
\Sigma(I) $, then by
  the above lemma there is gap between $I$ and $I+(\xi) \subset J$.
Another consequence of the lemma is the following corollary that lower
gaps in $\mathscr L$ can never be unique.

\begin{corollary}\label{C: gaps are never unique}
If $(\xi) \not \subset M$, then there is a gap in $\mathscr {L}$,
$N\subsetneq M+(\xi)$, with $N \not\subset M$.
Consequently, if $M\subsetneq J$ is a gap in $\mathscr {L}$, then
there is another gap $N\subsetneq J$
  in $\mathscr {L}$ with $M\neq N$.
\end{corollary}

\begin{proof}
Decompose the ideal $(\xi)$ into the sum of two non-comparable
principal ideals \linebreak
$(\xi) = (\rho) + (\eta)$, i.e. $(\eta)\not\subset(\rho)$ and
$(\rho)\not\subset(\eta)$,
  which implies that both are properly contained in $(\xi)$.
For this decomposition, inductively choose a sequence of indices
$n_1=1$ and $n_k \leq kn_k < n_{k+1}$
where $n_{k+1}$ is the smallest integer $j$ for which $\xi_j <
\frac{1}{k} \xi_{kn_k}$. Then
  define $\zeta_i := \frac{1}{k} \xi_{kn_k}$ for  $i\in [n_k,n_{k+1})$,
and $\eta_i=\zeta_i$ and $\rho_i = \xi_i$ for  $i\in [n_k,n_{k+1})$
for $k$ odd,
$\rho_i=\zeta_i$ and $\eta_i = \xi_i$ for  $i\in [n_k,n_{k+1})$ for $k$ even.
Then $\rho,\eta \leq \xi$, $\rho,\eta \in \co*$,
  and $(\frac{\xi}{D_k\rho})_{kn_k} =k$ for $k$ odd and
$(\frac{\xi}{D_k\eta})_{kn_k} =k$ for $k$ even,
from which it follows that $(\rho)\subsetneq (\xi)$,
  $(\eta)\subsetneq (\xi)$, and $(\eta)\not\subset(\rho)$,
$(\rho)\not\subset(\eta)$.
Moreover, $(\xi) = (\rho) + (\eta)$ since $\max\,(\rho,\eta) = \xi
\leq \rho + \eta \leq 2\xi$.

Now let $M_1$ (resp., $M_2$) be maximal among those ideals contained
in $M+(\xi)$
  that contain $(\rho)$ but do not contain $(\xi)$ (resp., that
contain $(\eta)$ but do not contain $(\xi)$).
Then each $M_i \subset M+(\xi)$ forms a gap by the maximality of the ideal $M_i$.
But if $M\supset M_1$ and $M\supset M_2$,
then $M \supset (\rho)$ and $M \supset (\eta)$ hence
  $M \supset (\eta) + (\rho) = (\xi)$ against the assumption, so at
least one $N:=M_i  \not \subset M$.
In particular, if $M\subsetneq J$ is a gap in $\mathscr {L}$, then
  $J=M+(\xi)$ for every principal ideal $(\xi)\not\subset M$ and for
any choice of such a principal ideal,
the pair $N \subsetneq J$ constructed above provides a lower gap for $J$ distinct from
$M \subsetneq J$.
\end{proof}

\noindent An immediate consequence of Lemma \ref{L: gaps} by choosing
$N=0$ in (ii) is:

\begin{corollary}\label{C: gaps princ}
Every principal ideal has lower gaps in $\mathscr {L}$.
\end{corollary}

\noindent Countably generated ideals may fail to have lower gaps, as
the following example shows.

\begin{example}\label{E: count without gaps}
Let $I$ be an ideal with generators $\eta^{(k)}$, where the sequences
$\eta^{(k)}$
  all satisfy the $\Delta_{1/2}$-condition and $\eta^{(k)} =
o(\eta^{(k+1)})$. Then $I$ has no lower gaps
in $\mathscr {L}$.
\begin{proof}
Assuming otherwise, by Lemma \ref{L: gaps}(ii), $I=N+(\xi)$ for some
principal ideal $(\xi)\not\subset N$.  Then $(\xi) \subset
(\eta^{(k_o)})$ for some $k_o$ and hence $ \xi = O( \eta^{(k_o)})$.
Without loss of generality, assume $\xi \leq \eta^{(k_o)} \leq
\eta^{(k_o+1)}$. Since
\[\eta^{(k_o+1)} \leq \zeta+KD_m\xi \leq  \zeta+KD_m\eta^{(k_o)} \leq
\zeta +K'\eta^{(k_o)}\] for some
$ \zeta \in \Sigma(N), K, K' > 0$, and $m\in \mathbb {N}$, and since
$\eta^{(k_o)} =o(\eta^{(k_o+1)})$, it follows
that $\eta^{(k_o+1)} \leq K'' \zeta$ for some $K''>0$. Hence $(\xi) \subset N$,
contradicting the assumption $(\xi)\not\subset N$.
\end{proof}
\end{example}

\noindent The fact that countably generated ideals and, in particular,
principal ideals never have upper gaps
  will follow from the more general fact that \textit{strongly soft
complemented} (ssc for short) ideals
never have upper gaps, which we prove in Proposition \ref{P: ssc} below.
An ideal $I$ is defined in \cite[Definition 4.4]{vKgW04-Soft} to be
ssc if for every sequence
$\eta^{(k)}\not\in \Sigma(I)$ there is a sequence of
integers $n_k$ for which if $\xi \in \co*$ and
$\xi_i \geq  \eta^{(k)}_i$ for all $1\leq i\leq n_k$, then $\xi
\not\in  \Sigma(I)$. Without loss of generality, $n_k$ is  increasing.
In \cite[Propositions 4.5, 4.6, 4.7, 4.11, 4.12, and 4.18]{vKgW04-Soft}
we show that many ``classical" ideals are ssc, among them countably generated
ideals, maximal symmetric norm ideals, Lorentz,
  Marcinkiewicz, and Orlicz ideals.

\begin{proposition}\label{P: ssc}
Let $I = \bigcap I_{\gamma} $ be the intersection of strongly soft-complemented
ideals $I_{\gamma}$.
\item[(i)] If $I \subsetneq I +(\xi) $, then there is a principal ideal
$(\eta)\subset (\xi)$ such that $I \subsetneq I + (\eta)\subsetneq I +(\xi)$.
\item[(ii)] I has no upper gaps in $\mathscr {L}$.
\end{proposition}

\begin{proof}
\item[(i)] First we prove the statement when $I$ itself is strongly
soft-complemented.
Since \linebreak
$\frac{1}{k}D_{\frac{1}{k}}\xi \not\in \Sigma(I)$
for all $k$
there is an increasing sequence of integers $n_k$ such that if $ \zeta \in \co*$ and
$\zeta_i \geq \frac{1}{k}D_{\frac{1}{k}}\xi$ for $ 1\leq i\leq n_k$,
 then  $ \zeta \not\in \Sigma(I)$. Define
$\eta_i = \frac{1}{k}(D_{\frac{1}{k}}\xi)_i$ for
$i \in (n_{k-1}, n_k] $; then $\eta \in \co*$ and  $\eta \leq \xi$. Since
$\frac{1}{j}D_{\frac{1}{j}}\xi\ge\frac{1}{k}D_{\frac{1}{k}}\xi$ for
every $j\le k$,
$\eta_i \ge \frac{1}{k}(D_{\frac{1}{k}}\xi)_i$ for all $i \in [1,
n_k] $ and hence by the strong soft-complementedness of
$I$, $\eta \not\in \Sigma(I)$. Thus $I \subsetneq I+(\eta)$.
For every $m\in \mathbb N$ and
for every $i > mn_m$ choose that $k$ for which
$n_{k-1} < \lceil \frac {i}{m} \rceil \leq n_k$.
It is easy to verify that $k\geq m+1$, hence
${k\lceil \frac {i}{m} \rceil} \geq i$, and thus $(D_m\eta)_i =
\eta_{ \lceil \frac {i}{m} \rceil}
=\frac{1}{k}\xi_{k \lceil \frac {i}{m} \rceil}
\le \frac{1}{k}\xi_i$, i.e., $D_m\eta = o(\xi)$ and hence $(\eta) \subset (\xi)$.
This  implies that
$(\xi) \not\subset I+(\eta)$ and hence that $I + (\eta)\subsetneq I
+(\xi) $. Indeed otherwise $\xi\leq \zeta+KD_m\eta$
for some $\zeta \in \Sigma(I)$,  $m\in \mathbb N$, and $K > 0$. Since
$D_m\eta = o(\xi)$, eventually
$\xi_i \leq 2\zeta_i$, that is $(\xi) \subset I$, against the hypothesis.

Now we prove the general case. Since $(\xi) \not\subset I$, then
$(\xi) \not\subset I_{\gamma}$ for some $I_{\gamma}$ in the collection.
Then by the first part of the proof there is a principal ideal $
(\eta)\subset (\xi)$ such that
$I_{\gamma} \subsetneq I_{\gamma} + (\eta)\subsetneq I_{\gamma}
+(\xi) $. In particular,
$ (\eta) \not\subset I_{\gamma}$ hence $ (\eta) \not\subset I$ and similarly,
$ (\xi) \not\subset I_{\gamma}+(\eta)$, hence
$ (\xi) \not\subset I+(\eta)$. Thus  $I \subsetneq I +
(\eta)\subsetneq I +(\xi) $.
\item[(ii)] Immediate from Lemma \ref{L: gaps}.
\end{proof}

\noindent In particular, Proposition \ref{P: ssc} applies to all ssc ideals (e.g., all countably generated ideals).\\

All the gaps proved here depend on the Hausdorff Maximality Principle
or its equivalent, Zorn's Lemma.
We do not know if the existence of gaps in $\mathscr L$ is equivalent to these. \\

Next we discuss density of the lattices of principal ideals and countably generated ideals.

\begin{corollary}\label{C: count gen upper density}
\item[(i)]  $\mathscr {L}_{\aleph_o}$ is upper dense in $\mathscr {L}$.
\item[(ii)] $\mathscr {PL}$ is upper dense in $\mathscr {L}$.
In particular, $\mathscr {L}_{\aleph_o}$  and $\mathscr {PL}$ have no gaps.
\end{corollary}

\begin{proof}
\item[(i)] If $I \in \mathscr {L}_{\aleph_o}$, $J \in \mathscr {L}$,
and $I \subsetneq J$, then there is a
$\xi \in \Sigma(J)\setminus \Sigma(I) $ and hence $I\subsetneq I+(\xi)$.
By \cite [Proposition 4.5] {vKgW04-Soft}, $I$ is ssc, so by Proposition \ref {P: ssc}(i), there is a principal ideal
$(\eta)$ such that $I \subsetneq I + (\eta)\subsetneq I +(\xi) \subset J$.
By adding $\eta$ to each of a countable collection of generators for
$I$, one easily sees that $I + (\eta) \in \mathscr {L}_{\aleph_o}$.
\item[(ii)] Use the same argument as in (i) and the fact that if $I
\in \mathscr {PL}$, then also $I + (\eta) \in \mathscr {PL}$.
\item[(iii)] Apply Proposition \ref{P: ssc}(i) for $I$ a single principal ideal.
\end{proof}

\noindent Corollary \ref {C: count gen upper density} implies:
\begin{corollary}\label{C: count gen and inters}
Let $I\in \mathscr L_{\aleph_o}$ and $J \in \mathscr {PL}$.
\item[(i)]
$\bigcap \{L \in \mathscr L \mid L \supsetneq I\}  = \bigcap \{L \in
\mathscr L_{\aleph_o} \mid L \supsetneq I\}  = I$
\item[(ii)]
$\bigcap \{L \in \mathscr L \mid L \supsetneq J\}  = \bigcap \{L \in
\mathscr {PL} \mid L \supsetneq J\}  = J$
\end{corollary}

Gaps in other lattices can be studied by similar tools as the following two simple examples illustrate.

If $I \subsetneq J$ are am-closed ideals,
then  $I \subsetneq I+(\xi) \subset J$ for some principal ideal
$(\xi)\subset J$. But then $(\xi)^-\subset J$, hence
  $I \subsetneq I+(\xi)^- \subset J$.
By \cite [Theorem 2.5]{vKgW04-Soft}, $I+(\xi)^-$ is am-closed and by
  \cite[Lemma 2.1(v)]{vKgW04-Soft}, the union M of a maximal chain of
am-closed ideals $I\subset L \subset I+(\xi)^-$ with $(\xi)
\not\subset L$ is also am-closed.
If  $ M \subset L= L^-\subset I+(\xi)^-$ and
$L\neq I+(\xi)^-$, i.e., $(\xi)^- \not\subset L$, then $(\xi)
\not\subset L$ since $L$ is am-closed. By the
maximality of the chain, $I\subset M \subsetneq I+(\xi)^-$ is a
gap in the lattice of am-closed ideals.
  Similar arguments show that  pairs of am-open ideals, am-$\infty$
closed ideals and
am-$\infty$ open ideals all contain gaps in their respective  ideal lattices.

If $(\xi)\in  \mathscr {SPL}$ and $M\subset (\xi)$ is a gap in
$\mathscr L$, then $M$ must be am-stable. Indeed, if not
$M\subsetneq M_a\subset (\xi)_a=(\xi)$ would imply $M_a=(\xi)$, hence
$\xi = O(\mu_a)$ for some $\mu\in\Sigma(M)$.
Since also $\mu_a = O(\xi)$, $\mu_a \asymp \xi$, hence $\mu_a$ is regular.
But then $\mu$ itself is regular by \cite[Theorem 3.10]{DFWW} and thus
$\mu \asymp \xi$, against the assumption that
  $M\neq (\xi)$.\\

Together with the upper density of $\mathscr {PL}$ in $\mathscr {L}$,
the following three
theorems which are the technical core of this article, will provide
the tools we need for
applications to arithmetic mean ideals. Recall that a sequence $\xi
\in \co*$ satisfies the $\Delta_{1/2}$-condition if
$\xi_n \leq M\xi_{2n}$ for some $M > 1$ and all $n$ and that
$\Delta_{1/2} \mathscr {PL}$
  is the lattice of principal ideals with $\Delta_{1/2}$ generators.

\begin{theorem}\label{T: Delta1/2-PL strong density and strong gaps}
$\Delta_{1/2} \mathscr {PL}$ is upper and lower dense in $\mathscr {PL}$.
\end{theorem}

\begin{proof}
First we prove lower density. Let $(\xi)$ be a principal ideal where
the generating sequence $\xi$ satisfies the
$\Delta_{1/2}$-condition and let $(\eta) \subsetneq (\xi)$. Then $\eta = O(\xi)$, so assume
without loss of generality that
$\eta \leq \xi$. Construct inductively two sequences of indices
$q_{k-1} < n_k < q_k $ as follows. Let
$M > 1$ be a bound for which $\xi_n \leq M\xi_{2n}$ for all $n$,
choose an integer $k_o \geq M$,
set $q_{k_o} = 1$, and assume the construction up to $k-1\geq k_o$.
Since $\xi \notin \Sigma((\eta))$
  and hence $\xi \ne O(\eta)$, choose $n_k > q_{k-1}$ so that
$\xi_{n_k} \geq kM^k\eta_{n_k}$
and let  $q_k$ be the largest integer $i$ for which $\xi_i \geq
\frac{1}{k}\xi_{2^kn_k}$. Since
$\xi_{2^{k+1}n_k}\geq \frac{1}{M}\xi_{2^kn_k} \geq
\frac{1}{k}\xi_{2^kn_k}$, it follows that $q_k\geq 2^{k+1}n_k$.
Now define the sequence
\[
\zeta_i :=
\begin{cases}
\xi_i &i \in (q_{k-1}, n_k] \\
\frac{1}{j}\xi_i & i\in (2^{j-1}n_k, 2^jn_k] \text{ and $1\leq j \leq k$} \\
\frac{1}{k}\xi_{2^kn_k}    &i\in (2^kn_k, q_k].
\end{cases}
\]
By construction, $\zeta$ is monotone nonincreasing and $\zeta \leq \xi$.
By assumption, $\eta_i \leq \xi_i = \zeta_i$ for  $i \in (q_{k-1},n_k]$ and
$\zeta_i \geq \frac{1}{k}\xi_{2^{k}n_k} \geq \frac{1}{kM^k}\xi_{n_k}
\geq \eta_{n_k} \geq \eta_i$
for $i \in (n_k,q_k]$.
Thus $\eta\leq \zeta$ and hence $(\eta)\subset (\zeta)\subset (\xi)$.
The following inequalities show that the sequence $\zeta$ satisfies
the $\Delta_{1/2}$-condition.

\[
\frac{\zeta_i}{\zeta_{2i}}=
\begin{cases}
  \frac{\xi_i}{\xi_{2i}}\leq M  & i\in (q_{k-1}, n_k]\\
\\
  \frac{\frac{1}{j}\xi_i}{\frac{1}{j+1}\xi_{2i}}\leq 2M &i\in (2^{j-1}n_k, 2^jn_k] \text{ and $1\leq j \leq k-1$}\\
\\
\frac{\frac{1}{k}\xi_i}{\frac{1}{k}\xi_{2^kn_k}} \leq \frac{\xi_{2^{k-1}n_k}}{\xi_{2^kn_k}}\leq M &i\in (2^{k-1}n_k, 2^kn_k]\\
\\
\frac{\frac{1}{k}\xi_{2^kn_k}}{\zeta_{2i}} \leq \frac{\xi_{q_k}}{\zeta_{2q_k}} = \frac{\xi_{q_k}}{\xi_{2q_k}} \leq  M &i\in (2^kn_k, q_k]
\end{cases}
\]
Thus $\frac{\zeta_i}{\zeta_{2i}}\leq 2M$ for all $i$. Finally $\zeta
\ne O(\eta)$ and $\xi \ne O(\zeta)$ since
$\left(\frac{\zeta}{\eta}\right)_{n_k} =
\left(\frac{\xi}{\eta}\right)_{n_k} \geq kM^k$
and $\left(\frac{\xi}{\zeta} \right)_{2^kn_k}=k$ for all $k$.
Since $\zeta$ satisfies the $\Delta_{1/2}$-condition, it follows that
$(\eta) \ne (\zeta) \ne (\xi)$,
which concludes the proof of lower density.

Next we prove upper density. Let $(\xi) \in \Delta_{1/2} \mathscr {PL}$, let
$(\xi) \subsetneq (\eta)$ and assume without loss of generality  that
$\xi \leq \eta$. Construct inductively two increasing sequences
  of integers $n_k,q_k$ with $2^{n_{k-1}} < q_k < 2^{n_k}$ as follows.
Let $M > 1$ be such that $\xi_n \leq M\xi_{2n}$
for all $n$, choose an integer $k_o \geq M$, set $n_{k_o}=0$, and
assume the construction up to $k-1 \geq k_o$.
Since $\eta \ne O(\xi)$, we can find infinitely many integers $r_k$
such that $\eta_{r_k} \geq kM^{k+1}\xi_{r_k}$. Let $n_k:=[\log_2{r_k}]$,
then
\[
\eta_{2^{n_k}} \geq \eta_{r_k}\geq kM^{k+1}\xi_{r_k}\geq
kM^k\xi_{\frac{r_k}{2}}\geq kM^k\xi_{2^{n_k}}.
\] Choose $r_k$ sufficiently large so that
$\xi_{2^{n_{k-1}+1}} \geq k\xi_{2^{n_k-k}}$ and let $q_k$ to be the
largest integer $i$ for which
$\xi_i \geq k\xi_{2^{n_k-k}}$. As $k\xi_{2^{n_k-k}} \geq
\frac{k}{M}\xi_{2^{n_k-k-1}} > \xi_{2^{n_k-k-1}}$,
because $\xi_i >0$ for all $i$ as $\xi$ satisfies the $\Delta_{1/2}$-condition,
it follows that $2^{n_{k-1}} < 2^{n_{k-1}+1} \leq q_k < 2^{n_k-k-1} < 2^{n_k}$.
  Now define the sequence
\[
\zeta_i :=
\begin{cases}
\xi_i            & i\in [2^{n_{k-1}},q_k] \\
k\xi_{2^{n_k-k}}    &i\in (q_k,2^{n_k-k}] \\
j\xi_i            &i\in (2^{n_k-j},2^{n_k-j+1}] \text { for $1 \leq j \leq k$.}
\end{cases}
\]
Then by construction, $\zeta \in \co*$ and $\xi \leq \zeta$.
Since $\zeta_i=\xi_i \leq \eta_i$ for $i\in [2^{n_{k-1}},q_k]$ and
\[\zeta_i \leq k\xi_{2^{n_k-k}} \leq kM^k\xi_{2^{n_k}}\leq
\eta_{2^{n_k}} \leq \eta_i ~\text{for}~i\in (q_k,2^{n_k}],\]
it follows that $\zeta \leq \eta$ and hence $(\xi)\subset
(\zeta)\subset (\eta)$.
The following inequalities show that $\zeta$ satisfies the
$\Delta_{1/2}$-condition.

\[
\frac{\zeta_i}{\zeta_{2i}}=
\begin{cases}
  \frac{\xi_i}{\zeta_{2i}}\leq\frac{\xi_i}{\xi_{2i}}\leq M  & i\in
[2^{n_{k-1}},q_k]\\
  \frac{k\xi_{2^{n_k-k}}}{\zeta_{2i}}\leq
\frac{\xi_{2^{n_k-k}}}{\xi_{2^{n_k-k+1}}}\leq M
&i\in (q_k,2^{n_k-k}]\\
\frac{j\xi_i}{(j-1)\xi_{2i}}\leq 2M
&i\in (2^{n_k-j},2^{n_k-j+1}] \text { for $2 \leq j \leq k$.}\\
  \frac{\xi_i}{\xi_{2i}} \leq  M
&i\in (2^{n_k-1},2^{n_k}]
\end{cases}
\]
For the last inequality, notice that $2^{n_k+1} \leq q_{k+1}$, which was proved above.
Finally, from $\eta_{2^{n_k}} \geq kM^k\xi_{2^{n_k}} =
kM^k\zeta_{2^{n_k}}$ and $\zeta_{2^{n_k-k}}
= k\xi_{2^{n_k-k}}$ it follows that $\eta \ne O(\zeta)$ and $\zeta
\ne O(\xi)$. Since both $\xi$ and $\zeta$ satisfy the $\Delta_{1/2}$-condition,   $(\xi) \neq (\zeta) \neq(\eta)$, which concludes the proof.
\end{proof}

\noindent As stated in (\cite[Section 2.4 (22)]{DFWW} and in
\cite[Corollary 4.15(i)]{vKgW02}, the
$\Delta_{1/2}$-condition is  equivalent to the Potter-type condition
  $\xi_n \geq C (\frac{m}{n})^p\xi_m$ for some $C > 0$ (necessarily $C
\leq 1$), $p \in \mathbb{N}$, and all
$n \geq m$. Although we did not need to employ explicitly this
condition in the above proof, similar
conditions characterizing regular and $\infty$-regular sequences will
be essential in the proofs of the next two theorems.

\begin{theorem}\label{T: SPL density in PL}
$\mathscr {SPL}$ is upper and lower dense in $\mathscr {PL}$.
\end{theorem}

\begin{proof}

Let $(\xi)$ be an am-stable principal ideal, i.e., one whose
generating sequence $\xi$ is regular, namely,
$\xi_a \leq M\xi$ for some $M>1$. Equivalently, $\xi$ satisfies the
Potter-type condition
$\xi_n \geq C(\frac{m}{n})^{p_o} \,\xi_m$ for some $0< C \leq 1$,
$0 < p_o < 1$ and all $n \geq m$ \cite[Theorem 3.10 and Remark 3.11]{DFWW}.
(See also \cite[Proposition 2.2.1]{BGT89}, \cite{Aljan-Aran77}).
Choose any $ p_o < p < 1$.

First we show the lower density of $\mathscr {SPL}$ in $\mathscr {PL}$.
Let $(\eta)\subsetneq (\xi)$ and assume without loss of generality
that $\eta \leq \xi$.
Construct inductively two increasing sequences of indices $q_{k-1} <
n_k < q_k $ as follows.
Choose an integer $k_o \geq C^{-\frac{1}{p-p_o}}$, set $ q_{k_o} =
1$, and assume
the construction up to $k-1\geq k_o$.
Since $\xi\neq O(\eta)$, choose $n_k > q_{k-1} $ so that $\xi_{n_k}
\geq k^p\eta_{n_k}$ and choose
$ q_k$  to be the largest integer $i$ for which $\xi_i\geq
k^{-p}\xi_{n_k}$. Then
$ \xi_{kn_k} \geq Ck^{-p_o}\xi_{n_k} > k^{-p}\xi_{n_k}$ and hence
$n_k < kn_k \leq q_k$.
Now define the sequence
\[
\zeta_i :=
\begin{cases}
\xi_i                                &\text{on $(q_{k-1},n_k]$} \\
\max\left(k^{-p},\, C(\frac{n_k}{i})^p\right)\xi_{n_k} &\text{on $(n_k,q_k]$}\\
\end{cases}
\]
As a further consequence of the Potter-type condition and of the
inequality just obtained,
$\zeta_{n_k} = \xi_{n_k} \geq \xi_{n_k+1} \geq \max (k^{-p},\,
C(\frac{n_k}{n_k+1})^p)\xi_{n_k}
= \zeta_{n_k+1}$.
Also, $\zeta_{q_k} \geq k^{-p}\xi_{n_k} > \xi_{q_k+1} = \zeta_{q_k+1}$. Thus
$\zeta$ is monotone nonincreasing. By the definition of $q_k$ and by
the Potter-type condition,
it follows that
$\zeta_i \leq \xi_i$ for all $i$. By assumption,
$\zeta_i = \xi_i \geq \eta_i$ for all $i \in (q_{k-1},n_k]$, while
$\zeta_i \geq k^{-p} \xi_{n_k} \geq \eta_{n_k} \geq \eta_i$ for all
$i \in (n_k,q_k]$.
Thus $ \eta \leq \zeta \leq \xi$ and hence $(\eta) \subset (\zeta)
\subset (\xi)$.

Next we prove that the sequence $\zeta$ is regular and hence the
principal ideal $(\zeta)$ is am-stable.
If $j\in\ (q_{k-1},n_k]$, then
$(\zeta_a)_j\leq\ (\xi_a)_j \leq M\xi_j = M\zeta_j$.
If $j \in (n_k, q_k]$, then
\begin{align*}
j (\zeta_a)_j &= \sum_{1}^{n_k}\zeta_i + \sum_{n_k+1}^j
\max\left(\frac{1}{k^p},\, C(\frac{n_k}{i})^p\right)\xi_{n_k} \\
&\leq  \sum_{1}^{n_k}\xi_i +
(j-n_k)\frac{1}{k^p}\xi_{n_k}+C\sum_{n_k+1}^j (\frac{n_k}{i})^p
\xi_{n_k}\\
&\leq  n_k(\xi_a)_{n_k} + \frac{j}{k^p}\xi_{n_k}+C\sum_{i=2}^j
(\frac{n_k}{i})^p \xi_{n_k}\\
&\leq Mn_k\xi_{n_k} + \frac{j}{k^p}\xi_{n_k} +
\frac{j}{1-p}C(\frac{n_k}{j})^p\xi_{n_k}\\
&\leq j\left(\frac{1}{k^p}\xi_{n_k} +
(\frac{1}{1-p}+\frac{M}{C})C(\frac{n_k}{j})^p\xi_{n_k}\right)\\
&\leq j(1+\frac{1}{1-p} +\frac{M}{C})\zeta_j.
\end{align*}
Thus $\zeta_a = O(\zeta)$, i.e., $\zeta$ is regular.
Since $\zeta_{n_k}=\xi_{n_k}  \geq k^p\eta_{n_k}$, it follows that
$\zeta \neq O(\eta)$.

>From the Potter-type inequality, $\xi_{kn_k} \geq Ck^{-{p_o}}\xi_{n_k} = Ck^{-(p_o-p)}k^{-p}\xi_{n_k}
=Ck^{p-p_o}\zeta_{kn_k}$ and thus $\xi \neq O(\zeta)$.
As $\zeta$ is regular and hence satisfies the $\Delta_{1/2}$ condition, we conclude that $(\eta) \subsetneq (\zeta) \subsetneq (\xi)$.

Now we prove the upper density of $\mathscr {SPL}$ in $\mathscr {PL}$.
Let  $M > 1$, $p_o<p<1$, and $0<C \le 1$ be as above and let $(\xi)$ be am-stable with $(\xi)\subsetneq (\eta)$ for some principal ideal $(\eta)$, and
assume without loss of generality
  that $\xi \leq \eta$. Construct inductively two increasing sequences
of indices
$n_{k-1} <  q_k < n_k $ as follows.
Choose  an integer $k_{o} \geq 2C^{-1/p}$, set $n_{k_o} = 1$, and
assume the construction up to $k-1\geq k_o$.
Since $\xi$ is regular and hence not summable, choose $q_k > n_{k-1}$ such that
$\sum_{n_{k-1}}^{q_k} \xi_i \geq n_{k-1} \xi_1$.
Since $\eta \neq O(\xi)$, choose an integer
$n_k$ such that $\eta_{n_k} \geq k^p \xi_{n_k}$ and $\xi_{n_k} \leq
k^{-p}\xi_{q_k}$.
By increasing if necessary $q_k$, assume without loss of generality that $q_k$
is the largest integer $i$ such that $\xi _{i} \geq k^p \xi_{n_k}$.
Clearly, $q_k < n_k$. Next, define the sequence
\begin{align*}
\zeta_i =
  \begin{cases}
     \xi_i &\text{$i \in [n_{k-1},q_k]$}\\
     \min \left(k^p, \frac{1}{C} (\frac{n_k}{i})^p \right) \xi_{n_k}
&\text{$i \in (q_k,n_k)$}
  \end{cases}
\end{align*}
By definition, $\zeta_{q_k} = \xi_{q_k} \geq k^p \xi_{n_k} \geq
\zeta_{q_k+1}$ and since
$k^p > k_o^p \geq \frac{2^p}{C} >
\frac{1}{C}\left(\frac{n_k}{n_k-1}\right)^p$,  it follows that
$\zeta_{n_k-1}=\frac{1}{C}(\frac{n_k}{n_k-1})^p\xi_{n_k} > \xi_{n_k}
= \zeta_{n_k}$.
Thus $\zeta$ is monotone nonincreasing.
Since $\xi_i = \zeta_i $ for $i \in [n_{k-1}, q_k]$ and $\xi_{i} < \zeta_{i}$
for $i \in (q_k, n_k)$ because $\xi_{i} \leq \frac{1}{C} \left(
\frac{n_k}{i} \right)^{p_{o}} \xi_{n_k}
< \frac{1}{C} \left( \frac{n_k}{i} \right)^p \xi_{n_k}$ and $\xi_{i}
< k^p\xi_{n_k}$
  by the definition of $q_k$, it follows that
$\xi \leq \zeta$. Since $\zeta_{i} = \xi_i \leq \eta_{i}$ for $i \in
[n_{k-1}, q_k]$
and also $\zeta_{i} \leq k^p\xi_{n_k} \leq \eta_{n_k} \leq \eta_{i}$
for $i \in (q_k, n_k)$, it follows that
  $ \zeta \leq \eta$. Therefore $(\xi) \subset (\zeta) \subset (\eta)$.

Next we prove that the sequence $\zeta$ is regular and hence the
principal ideal $(\zeta) $ is am-stable.
If $j \in (q_k, n_k)$, then recalling that by definition $\xi_1=\zeta_1$, $\xi\le \zeta$, and that $\sum_{n_{k-1}}^{q_k} \xi_i \geq n_{k-1} \xi_1$,
\begin{align*}
j(\zeta_a)_j &= \sum_{1}^{n_{k-1}-1} \zeta_i + \sum_{n_{k-1}}^{q_k} \xi_i +
\sum_{q_k+1}^{j} \min \left(k^p,\,\frac{1}{C} (\frac{n_k}{i})^p
\right) \xi_{n_k} \\
&\le n_{k-1}\xi_1 + \sum_{n_{k-1}}^{q_k} \xi_i +
\sum_{2}^{j} \min \left(k^p,\,\frac{1}{C} (\frac{n_k}{i})^p \right) \xi_{n_k}\\
&\leq 2 \sum_1^j \xi_i+\min \left(\sum_{2}^{j}k^p,\,\sum_{2}^{j}
\frac{1}{C} (\frac{n_k}{i})^p\right)\xi_{n_k}\\
&\leq 2 j (\xi_a)_j  + \min \left(jk^p
,\,\frac{j}{1-p}\frac{1}{C}(\frac{n_k}{j})^p \right)\xi_{n_k}\\
&\le 2 Mj \zeta_j  + \frac{j}{1-p}\min \left(k^p
,\,\frac{1}{C}(\frac{n_k}{j})^p \right)\xi_{n_k}\\
&= (2M+\frac{1}{1-p})j\zeta_j.
\end{align*}
\noindent If $j \in [n_k, q_{k+1}]$, then by using the above inequality and the definition of $\zeta$,
and recalling that  $\zeta_{n_k-1}=\frac{1}{C} \left( \frac{n_k}{n_k - 1} \right)^p \xi_{n_k}$, we have
\begin{align*}
j(\zeta_{a})_{j} &= (n_k - 1)(\zeta_a)_{n_k-1} + \sum_{n_k}^j \xi_i \leq (2M + \frac{1}{1 - p})(n_k - 1)\zeta_{n_k-1} + \sum_1^j \xi_i\\
&= (2M + \frac{1}{1 - p}) \frac{n_k }{C} \left( \frac{n_k}{n_k - 1} \right)^{p-1} \xi_{n_k} + \sum_1^j \xi_i
  \leq \frac{1}{C}\left ((2M + \frac{1}{1 - p})n_k\xi_{n_k} +
\sum_1^j \xi_i\right)\\
& \leq \frac{1}{C}(2M + \frac{1}{1 - p}+1) j (\xi_a)_j
  \leq \frac{1}{C}(2M + \frac{1}{1 - p}+1) Mj \xi_j
  =\frac{1}{C}(2M + \frac{1}{1 - p}+1) Mj\zeta_j.
\end{align*}
\noindent Thus $\zeta_{a} = O(\zeta)$, i.e., $\zeta$ is regular and hence the
principal ideal $(\zeta)$ is am-stable.
Since $\eta_{n_k} \geq k^p\xi_{n_k} = k^p\zeta_{n_k}$ it follows
that $\eta \neq O(\zeta)$.
Set $m_k= \lceil \frac{n_k}{kC^{1/p}} \rceil$, the smallest integer
majorizing $\frac{n_k}{kC^{1/p}}$.
Since $\frac{m_k}{n_k}< \frac{1}{kC^{1/p}}+\frac{1}{n_k}<1$ because
$k > k_o$, and using the inequality
$\xi_{m_k} \le \frac{1}{C} ( \frac{n_k}{m_k})^{p_o} \xi_ {n_k}
< \frac{1}{C} (\frac{n_k}{m_k} )^p \xi_{n_k}\leq k^p\xi_{n_k}$, it
follows that $q_k < m_k < n_k$.
By the same inequalities,
\[
\zeta_{m_k}
= \frac{1}{C} (\frac{n_k}{m_k} )^p\xi_{n_k}
= (\frac{n_k}{m_k} )^{p-{p_o}}\frac{1}{C} ( \frac{n_k}{m_k})^{p_o} \xi_ {n_k}
>   \left(\frac{1    }{kC^{1/p}}+\frac{1}{n_k}\right)^{{p_o}-p}\xi_{m_k}.
\]
and hence $\zeta \neq O(\xi)$. 
As $\zeta$ is regular and hence satisfies  
the $\Delta_{1/2}$-condition, it follows that $(\xi) \subsetneq (\zeta) \subsetneq  (\eta)$, which concludes the proof.
\end{proof}

\noindent Now we consider am-$\infty$ stable principal ideals, i.e.,
ideals with $\infty$-regular generating sequences.
The main tool for proving the upper and lower density of
$S_{\infty}PL$ in $\mathscr {PL}$ is the Potter-type inequality
for $\infty$-regular sequences \cite[Theorem 4.12]{vKgW04-Traces} 
(see proof of Theorem \ref {T: SinftyPL density}).
Similar to the am-case, but much less trivial,
  a sequence $\xi$ is $\infty$-regular,
i.e., $(\xi)_{a_\infty} =(\xi)$ if and only if $\xi_{a_\infty} =
O(\xi)$ (see [ibid.]).
The main technical complication with respect to the arithmetic mean
case is that $\infty$-regular sequences
  in general do not satisfy the $\Delta_{1/2}$-condition
(e.g., see \cite [Example 4.5(ii)] {vKgW04-Traces}). Thus $D_m$
considerations are unavoidable,
i.e., to prove that the inclusion of two principal ideals $(\eta)
\subset (\xi)$ is proper,
  it is necessary to show that $\xi \neq O(D_m\eta)$ for all positive
integers m.
  We first need the following lemmas.

\begin{lemma}\label{L: am infty under}
If $(\xi) \supsetneq F$, then $(\xi)\supsetneq (\eta) \supsetneq F$ for some
am-$\infty$ stable principal ideal $(\eta)$.
\end{lemma}
\begin{proof}
Assume without loss of generality that $\xi_1 < 1$ and define $\eta
:= \,<2^{-j} \prod_{j}^{j^2}\xi_i>$. Clearly, the sequences
$<\prod_{j}^{j^2}\xi_i>$ and hence $\eta$ are monotone decreasing,
$\eta \leq\xi$ and $\eta_i > 0$ for all $i$.
Furthermore, for every $m$ and every $n \geq m$, $({D_m}\eta)_{mn} = \eta_n
= 2^{-n}\prod_{n}^{n^2}\xi_i \ < 2^{-n}\xi_{n^2}  \le 2^{-n}\xi_{mn}$, hence
$\xi \neq O(D_m\eta)$. Thus $(\xi)\supsetneq (\eta) \supsetneq F$.  Moreover,
\[
(\eta_{a_\infty})_n = \frac {1}{n} \sum_{n+1}^\infty
(2^{-j}\prod_{j}^{j^2}\xi_i)
  \leq \frac {1}{n}\left(\prod_{n}^{n^2}\xi_i\right) \sum_{n+1}^\infty 2^{-j} =
\frac {1}{n} \eta_n,
\]
hence $\eta$ is $\infty$-regular by
\cite [Theorem 4.12]{vKgW04-Traces} and therefore $(\eta)$ is am-$\infty$ stable.
\end{proof}

\begin{lemma}\label{L: density of summable}
If $\{0\}\ne (\xi)\subsetneq (\eta)$ and $\xi$ is summable,
then $(\xi) \subsetneq (\zeta) \subsetneq (\eta)$ for some summable $\zeta$.
\end{lemma}
\begin{proof}
By Corollary \ref{C: count gen upper density}(ii), it is enough to find a  summable $\zeta$ such that $(\xi) \subsetneq (\zeta) \subset (\eta)$. 
The case when $\xi$ is finitely supported, i.e., $(\xi)=F$ is trivial: it is enough to construct a summable but not finitely supported sequence
$\zeta$ with $(\zeta) \subset (\eta)$, e.g., $<\frac{\eta_n}{2^n}>$.
Thus assume $\xi_i > 0$ for all $i$, and without loss of generality, assume that $\xi \le \eta$.
Construct inductively two increasing sequences of indices $n_{k-1} < q_k < n_k$ starting with $n_1=1$
and satisfying the following three conditions:
$\eta_{n_k} \geq k\xi_{\lceil\frac{n_k}{k}\rceil}$, $\xi_{\lceil\frac{n_k}{k}\rceil} < \frac{1}{k}\xi_{n_{k-1}}$,
and $n_k\xi_{{\lceil\frac{n_k}{k}\rceil}}\leq k^{-3}$.
Indeed, since $\eta \neq O(D_k\xi)$ for all $k$, there are infinitely many indices $i$ such that $\eta_i\geq k\xi_{\lceil\frac{i}{k}\rceil}$,
and among those indices, we can choose $n_k$ that satisfies the second condition since $\xi_i \to 0$ and the third since 
$i\xi_i \to 0$ by the summability of $\xi$. 
Set $q_k$ to be the largest index $i$ such that $\xi_i \geq k\xi_{{\lceil\frac{n_k}{k}\rceil}}$. 
Then $n_{k-1} \leq q_k < {\lceil\frac{n_k}{k}\rceil} < n_k$. 
Next, for a sequence $\zeta \in \text{c}_\text{o}^*$ with $\xi \leq\zeta\leq\eta$ define:
\[
\zeta_i :=
\begin{cases}
\xi_i       &\text{on $(n_{k-1},q_k)$}\\
k\xi_{\lceil\frac{n_k}{k}\rceil}
        &\text{on $[q_k,n_k]$}\\
\end{cases}
\]
Then $\zeta\neq O(D_m\xi)$ for all $m$ since $\zeta_{n_k}
= k\xi_{\lceil\frac{n_k}{k}\rceil} = k(D_k\xi)_{n_k} \geq
k(D_m\xi)_{n_k}$ for all $k\geq m$. Thus
$\xi \subsetneq\zeta\subset\eta $. Finally, $\zeta$ is summable because
\[
\sum_{i=1}^{\infty}\zeta_i
< \sum_{i=1}^{\infty}\xi_i + \sum_{k=1}^{\infty}\sum_{i=q_k}^{n_k}
k\xi_{\lceil\frac{n_k}{k}\rceil}
< \sum_{i=1}^{\infty}\xi_i
+\sum_{k=1}^{\infty}kn_k\xi_{\lceil\frac{n_k}{k}\rceil}
\leq \sum_{i=1}^{\infty}\xi_i + \sum_{k=1}^{\infty}k^{-2}<\infty.
\]
\end{proof}

\noindent In the terminology of this paper, the above lemma states
that the lattice of principal ideals contained
in $\mathscr L_1$ (i.e., with trace class generators) is upper dense
in $\mathscr {PL}$.
By the lack of gaps in $\mathscr {PL}$ (Corollary \ref{C: count gen upper density}(ii),
the lattice of principal ideals contained
in $\mathscr L_1$ is also lower dense in $\mathscr {PL}$.

\begin{theorem} \label{T: SinftyPL density}
$\mathscr {S_{\infty}PL}$ is upper and lower dense in $\mathscr {PL}$.
\end{theorem}
\begin{proof}
Let $(\xi)$ be an am-$\infty$ stable principal ideal, i.e., $\xi$ is
an $\infty$-regular sequence, in particular, $\xi$ is summable. 
By \cite[Theorem 4.12]{vKgW04-Traces}, $\xi_{a_\infty} \leq M\xi$ for some $M>0$, and
by the same theorem, $\xi_n
\leq C(\frac{m}{n})^{p_o} \,\xi_m$ for some $C \geq 1$, $p_o > 1$ and
all $n \geq m$. Choose  any $1<p<p_o$.

First we prove the lower density of $\mathscr {S_{\infty}PL}$ in
$\mathscr {PL}$.
Let $(\eta)\subsetneq (\xi)$. Then  $\eta \ \leq KD_m\xi$ for some
positive integer $m$ and some $K>0$
and since $(KD_m\xi)=(\xi)$, without loss of
generality we can assume that $\eta\leq \xi$.
The case where $\eta$ is finitely supported, i.e. $(\eta) = F$
follows from Lemma \ref {L: am infty under}
above, so assume that $\eta_i > 0 $ for all $i$.
We construct inductively two increasing sequences of indices
$q_{k-1}<m_k<q_k$ as follows. Choose an integer
  $k_o \geq M+1$, set $q_{k_o}=1$ and
assume the construction up to $ k-1 \geq k_o$. Using the fact that for
all $m \in \mathbb {N}$, $\xi \neq O(D_m\eta)$, choose $n_k \geq
2kq_{k-1}$ such that $\xi_{n_k} \geq k^R(D_k\eta)_{n_k} = k^R
\eta_{m_k}$ where $R:=\frac{2pp_o}{p_o-p}$ and
$m_k:= {\lceil\frac{n_k}k\rceil}$ is the smallest integer majorizing
$\frac{n_k}k$; thus $m_k \geq 2q_{k-1}$.
Then define $q_k$ to be the largest integer $i$ for which $\xi_i \geq
\eta_{m_k}$; thus
$m_k < n_k \leq  q_k$. Next define the sequence
\[
\zeta_i :=
\begin{cases}
\xi_i
        &\text{on $(q_{k-1},m_k)$}\\

\min\,(\xi_i, (\frac{q_k}{i})^p \eta_{m_k})         &\text{on $[m_k,q_k]$}.\\

\end{cases}
\]

Since $\zeta_{m_k-1} = \xi_{m_k-1} \geq \xi_{m_k} \geq \zeta_{m_k}$ and
$\zeta_{q_k} = \eta_{m_k} > \xi_{q_k+1} = \zeta_{q_k+1}$,  it follows
that  $\zeta$ is monotone non-increasing. Clearly,
$\zeta_i \leq \xi_i$ for all i, $\zeta_i \geq \eta_i$ for $i \in
(q_{k-1},m_k)$, and
$\zeta_i \geq \zeta_{q_k} = \eta_{m_k} \geq \eta_i$ for $i \in
[m_k,q_k]$. Thus $\eta \leq \zeta\leq \xi$  and hence
$ (\eta) \subset (\zeta) \subset (\xi)$.

Now we prove that $\zeta$ is am-$\infty$ regular by showing that
$\zeta_{a_\infty} = O(\zeta)$ (see \cite [Theorem 4.12]{vKgW04-Traces}).
If $j\in [m_k,q_k]$, then the Potter-type inequality implies that
$\xi_j \geq \frac{1}{C}(\frac{q_k}{j})^{p_o}\xi_{q_k} \geq
\frac{1}{C}(\frac{q_k}{j})^p\eta_{m_k}$, hence  $\zeta_j \geq
\frac{1}{C}(\frac{q_k}{j})^p\eta_{m_k}$
and thus $\zeta_j\leq (\frac{q_k}{j})^p\eta_{m_k} \leq C\zeta_j$.
Thus, using the fact that
$\xi_{q_k+1} < \eta_{m_k}$ and $q_k\geq k_o \geq M+1$, by definition
of $q_k$ and $k_o$,
\begin{align*}
j(\zeta_{a_\infty})_j &\leq \sum_{j+1}^{q_k} \zeta_i +
\sum_{q_k+1}^{\infty} \xi_i
\leq \sum_{j+1}^{q_k} (\frac{q_k}{i})^p \eta_{m_k} + \xi_{q_k+1}+
(q_k+1)(\xi_{a_\infty})_{q_k+1}\\
&\leq \frac{j}{p-1}(\frac{q_k}{j})^p \eta_{m_k} + \xi_{q_k+1}+
M(q_k+1)\xi_{q_k+1}
\leq (\frac{j}{p-1}(\frac{q_k}{j})^p  + (M+1)q_k)\eta_{m_k}\\
&\leq (\frac{1}{p-1}+M+1 ) j(\frac{q_k}{j})^p \eta_{m_k}
\leq C(\frac{1}{p-1}+M+1 ) j\zeta_j.
\end{align*}
If $j\in (q_{k-1},m_k)$, then $(\zeta_{a_\infty})_j \leq
(\xi_{a_\infty})_j \leq M \xi_j = M \zeta_j$.
Thus $\zeta_{a_\infty}= O(\zeta ),$ which proves
the am-$\infty$ regularity of $\zeta$.
It remains to prove that both inclusions in $ (\eta) \subset (\zeta)
\subset (\xi)$ are proper.
By the Potter-type condition, $ \frac{q_k}{n_k} \leq
(C\frac{\xi_{n_k}}{\xi_{q_k}})^\frac{1}{p_o}$ and hence
\begin{align*}
\zeta_{m_k} &\leq (\frac{q_k}{m_k})^p \eta_{m_k}
\leq k^p (\frac{q_k}{n_k})^p \eta_{m_k}
\leq C^\frac{p}{p_o} k^p (\frac{\xi_{n_k}}{\xi_{q_k}})^\frac{p}{p_o} \eta_{m_k}
\leq C^\frac{p}{p_o}k^p(\frac{\xi_{n_k}}{\eta_{m_k}})^\frac{p}{p_o}
\eta_{m_k} \\
&=
C^\frac{p}{p_o}k^p(\frac{\eta_{m_k}}{\xi_{n_k}})^{(1-\frac{p}{p_o})}\xi_{n_k}
\leq C^\frac{p}{p_o}k^{p-R(1-\frac{p}{p_o})}\xi_{n_k}
= C^\frac {p}{p_o}  k^{-p} \xi_{n_k}.
\end{align*}
Thus $\xi_{n_k} \geq C^{-\frac{p}{p_o}}k^p(D_k \zeta)_{n_k}$
for all
$k \geq k_o$, and hence $\xi \neq O(D_m\zeta)$
for any $m$, i.e., $(\xi) \not\subset (\zeta)$. Finally, recalling that
$(\frac{q_k}{j})^p\eta_{m_k} \leq C\zeta_j$ for all $j\in [m_k,q_k]$ and that
$n_k \ge \frac{k}{2}m_k$ and hence $\frac{k^2}{2}m_{k^2} \le n_{k^2} \le q_{k^2}$ we
have

\[
\zeta_{km_{k^2}}
\geq \frac{1}{C}(\frac{q_{k^2}}{km_{k^2}})^p \eta_{m_{k^2}}
\geq \frac{1}{C}(\frac{n_{k^2}}{km_{k^2}})^p \eta_{m_{k^2}}
\geq \frac{1}{C}(\frac{k}{2})^p \eta_{m_{k^2}}
\]
and hence $\zeta \neq O(D_m\eta)$ for any m.
Thus $(\zeta) \not\subset (\eta)$,  completing the proof of
lower density.

Now we prove  the upper density of $\mathscr {S_{\infty}PL}$ in
$\mathscr {PL}$.
Let $(\eta)\supsetneq (\xi)$ and, as in the first part of the proof,
assume without loss of
generality that $\eta\geq \xi$ and additionally that $\eta_1 =\xi_1$.
By Lemma \ref {L: am infty under},
assume that $\xi_i > 0$ for all $i$ and by Lemma \ref {L: density of summable}
assume also that $\eta$ is summable.

We construct inductively two increasing sequences of indices
$n_{k-1}<q_k<n_k$ as follows. Choose $n_1=1$ and
assume the construction up to $k-1$. For all positive integers $n$
define  $q(n):=\max{\{ i\in \mathbb {N} \mid \xi_i \geq \eta_n}\}$;
then clearly $q(n)\uparrow \infty$. Using the fact that
$\eta \neq O(D_k\xi)$ and that $\eta$ is summable, choose $n_k$ so
that  $\eta_{n_k} \geq k^R(D_k\xi)_{n_k}$
and $\sum _{q(n_k)}^\infty \eta_i \leq \sum _{n_{k-1}}^\infty \xi_i$ where
$R:=\frac{2pp_o}{p_o-p}$. Set $q_k : = q(n_k)$ and
$m_k:= {\lceil\frac{n_k}k\rceil}$, the smallest integer majorizing
$\frac{n_k}k$.
  Because of the inequalities
$\xi_{q_k} \geq\eta_{n_k} \geq k^R \xi_{m_k} > \xi_{m_k}$ and the
summation condition, it follows
that $n_{k-1}<q_k <m_k < n_k$.
Next, define the sequence
\[
\zeta_i :=
\begin{cases}
\xi_i
        &\text{on $(n_{k-1},q_k]$}\\

\max\,(\xi_i, (\frac{q_k}{i})^p \eta_{n_k})         &\text{on $[q_k,n_k]$}\\

\end{cases}
\]
Since $\zeta_{n_k} \geq \xi_{n_k}  \geq \xi_{n_{k+1}} =
\zeta_{n_k+1}$ we see that
$\zeta$ is monotone nonincreasing. By definition, $\xi_i \leq
\zeta_i$ for all $i$,
$\zeta_i \leq \eta_i$ for all $i \in (n_{k-1},q_k]$, and $\zeta_i
\leq \eta_{n_k} \leq \eta_i$
for all $i\in [q_k,n_k]$. Thus $\xi \leq \zeta\leq\eta$ and hence
$(\xi) \subset (\zeta) \subset (\eta)$.
Now we prove that $\zeta$ is am-$\infty$-regular. If $j\in [q_k,n_k]$,  then
\begin{align*}
j(\zeta_{a_{\infty}})_j &= \sum _{j+1}^\infty \zeta_i \leq \sum
_{j+1}^{n_k}\xi_i + \sum _{j+1}^{n_k} (\frac{q_k}{i})^p \eta_{n_k}
+\sum _{n_k+1}^{q_{k+1}}\xi_i + \sum _{q_{k+1}}^\infty \eta_i\\
&\leq \sum _{j+1}^{q_{k+1}}\xi_i + \sum _{j+1}^{n_k}
(\frac{q_k}{i})^p \eta_{n_k} + \sum _{n_k}^\infty \xi_i
\leq 2\sum _{j+1}^\infty \xi_i + \frac{j}{p-1}(\frac{q_k}{j})^p\eta_{n_k} \\
&\leq 2Mj\xi_j + \frac{j}{p-1}(\frac{q_k}{j})^p\eta_{n_k}
\leq j(2M+\frac{1}{p-1})\zeta_j.
\end{align*}
If $j \in (n_{k-1},q_k)$, then by applying the inequality just
obtained and the Potter-type inequality,
\begin{align*}
j(\zeta_{a_{\infty}})_j &= \sum _{j+1}^{q_k} \xi_i + \sum
_{q_k+1}^{\infty} \zeta_i
\leq  \sum _{j+1}^\infty \xi_i + (2M+\frac{1}{p-1})q_k\xi_{q_k}
\leq Mj\xi_j + (2M+\frac{1}{p-1})Cq_k(\frac{j}{q_k})^p\xi_j\\
&\leq Mj\xi_j + (2M+\frac{1}{p-1})Cj\xi_j
\leq  C(3M+\frac{1}{p-1})j\zeta_j.
\end{align*}
Thus $\zeta_{a_\infty} = O(\zeta)$ and hence $\zeta$ is am-$\infty$ regular.
Now we prove that the ideal inclusions are proper.  By the
Potter-type inequality
$\frac{q_k+1}{m_k} \geq
(\frac{1}{C}\frac{\xi_{m_k}}{\xi_{q_k+1}})^{\frac{1}{p_o}}$, therefore
\begin{align*}
\zeta_{n_k} &\geq \left(\frac {q_k}{n_k}\right)^p\eta_{n_k}
\geq (2k)^{-p}\left(\frac{q_k+1}{m_k}\right)^p\eta_{n_k}
\geq (2k)^{-p}
\left(\frac{1}{C}\frac{\xi_{m_k}}{\xi_{q_k+1}}\right)^{\frac{p}{p_o}}\eta_{n_k}\\
&\geq (2k)^{-p}C^{-\frac{p}{p_o}}
\left(\frac{\xi_{m_k}}{\eta_{n_k}}\right)^{\frac{p}{p_o}}\eta_{n_k}
=(2k)^{-p}C^{-\frac{p}{p_o}}
\left(\frac{\eta_{n_k}}{\xi_{m_k}}\right)^{1-\frac{p}{p_o}} \xi_{m_k}
\geq 2^{-p}C^{-\frac{p}{p_o}} k^{-p+R(1-\frac{p}{p_o})} \xi_{m_k}\\
&=2^{-p}C^{-\frac{p}{p_o}} k^p \xi_{m_k}.
\end{align*}
Thus $\zeta \neq O(D_m\xi)$ for all m and hence $(\xi) \neq (\zeta)$. Finally,
$q_{k^2} < m_{k^2} <km_{k^2} < k^2m_{k^2}  \le n_{k^2}$ and hence, again by the
Potter-type condition,
\[
\zeta_{km_{k^2}} = \max\left(\xi_{km_{k^2}},
(\frac{q_{k^2}}{km_{k^2}})^p\eta_{n_{k^2}}\right)
\leq  \max\left(C(\frac{q_{k^2}+1}{km_{k^2}})^{p_o}\xi_{q_{k^2}+1},
k^{-p}\eta_{n_{k^2}}\right)
\leq Ck^{-p}\eta_{n_{k^2}}.
\]
Thus $\eta \neq O(D_m\zeta)$ for all m and so $(\eta)
\neq (\zeta)$, concluding the proof.
\end{proof}

Since $\mathscr {PL}$ is upper but not lower dense in $\mathscr {L}$ by
Corollaries \ref{C: count gen upper density} and \ref{C: gaps princ},
immediate consequences of
Theorems \ref{T: Delta1/2-PL strong density and strong gaps},
\ref{T: SPL density in PL}, and \ref{T: SinftyPL density} are:

\begin{corollary} \label{C:density in L}
\item[(i)] $\mathscr {PL}$, $\Delta_{1/2} \mathscr {PL}$, $\mathscr
{SPL}$, and $\mathscr {S_{\infty}PL}$
are upper dense in $\mathscr {L}$ but they all have lower gaps in
$\mathscr {L}$.
\item[(ii)] $\mathscr {PL}$, $\Delta_{1/2} \mathscr {PL}$, $\mathscr
{SPL}$, and $\mathscr {S_{\infty}PL}$
   have no gaps.
\end{corollary}

Other consequences of the same theorems and the Potter-type
conditions already mentioned are:

\begin{corollary}\label{C:inclusions}
 \item[(i)] Every principal ideal is contained in a principal ideal
from $\mathscr {SPL}$
and if it is strictly larger than $F$ it contains a principal ideal
in $\mathscr {S_{\infty}PL}$ that is strictly
  larger than $F$.
\item[(ii)] A principal ideal $(\eta)$ contains a principal ideal in
$\Delta_{1/2} \mathscr {PL}$ if and only
  if $\omega^p = O(\eta )$ for some $p>0$; it contains a principal
ideal in $\mathscr {SPL}$  if and only if
$\omega^p = O(\eta )$ for some $0 < p < 1$; and it is contained in  a
principal ideal in $\mathscr
{S_{\infty}PL}$ if and only if $\eta = O(\omega^p)$ for some $ p > 1$.
\end{corollary}
\begin{proof}
\item[(i)] Assume without loss of generality that the principal ideal $I$ has generator $\eta$ with $\eta_1=1$ and
hence $||\eta_{a^m}|| = 1$  for the $c_o$-norm for all $m \in \mathbb
N$.
Thus the series $\sum_{m=1}^{\infty}2^{-m} \eta_{a^m} := \zeta$
converges in the $c_o$-norm, hence weakly, and thus $\zeta \in
\co*$ and $(\eta) \subset (\zeta)$. By definition,
\[
(\zeta_a)_n= \frac{1}{n}\sum_{j=1}^{n}\sum_{m=1}^{\infty} 2^{-m}(\eta_{a^m})_j
=\sum_{m=1}^{\infty}2^{-m} \frac{1}{n}\sum_{j=1}^{n}(\eta_{a^m})_j
=\sum_{m=1}^{\infty} 2^{-m}(\eta_{a^{m+1}})_n
\leq 2\zeta_n
\]
and hence $\zeta$ is regular.

When $I\neq F$, by Lemma \ref {L: am infty under} there is an ideal $
J \in \mathscr {S_{\infty}PL}$ with
$F\subsetneq J \subsetneq I $.
\item[(ii)] It is immediate to verify that the sequence $\omega^p$
satisfies the $\Delta_{1/2}$-condition
for every $p >0$, is regular if and only if $0<p<1$ and is
$\infty$-regular if and only if $p > 1$.
This establishes the sufficiency in (ii).
The necessity is a direct consequence of the Potter-type conditions
(in all cases taking $m=1$) that we have already quoted in this
section.
\end{proof}

\noindent By the lack of gaps in $\mathscr
{PL}$ and by Theorems \ref{T: Delta1/2-PL strong density and strong gaps},
\ref{T: SPL density in PL} and \ref{T: SinftyPL density}, the ideals in
(i) and (ii)  are of course never unique.\\

If $L$ is an am-stable ideal, $I$ is principal, and $I \subset L$, then while Corollary \ref{C:inclusions}(i) ensures that there are principal am-stable ideals $J \supset I$, it may be impossible to find any such $J$ also contained in $L$, as the following example shows. The same holds in the  am-$\infty$ case. \\

Recall from Section \ref{S:2} that the am (resp., am-$\infty$) stabilizer $st^a(I)$ (resp., $st_{a_\infty}(I)$) is the smallest
am-stable ideal containing $I$ (resp., the largest am-$\infty$ stable ideal contained in $I$).

\begin{example}\label{E:X}
\item[(i)]  Let $\xi$ be an irregular sequence. Then $L=st^a((\xi))$
is am-stable but there is no intermediate
am-stable principal ideal $(\xi) \subset J \subset L$.
\item [(ii)] $st_{a_{\infty}}((\omega))$ is am-${\infty}$ stable but
it contains principal ideals
  that are not contained in any am-${\infty}$ stable principal ideal.
\end{example}
\begin{proof}
\item[(i)] Assume by contradiction that there is an am-stable
principal ideal $J= (\eta)$ with $(\xi) \subset J \subset L$.
Since $J \supset (\xi)$ and $J$ is am-stable, then $J \supset L$,
hence $J=L = \bigcup_n(\xi_{a^n})$.Then $\eta \in
\Sigma((\xi_{a^n}))$ for some $n\in \mathbb N$.
On the other hand, $(\xi_{a^n})\subset (\eta)$, hence $\xi_{a^n}\asymp \eta$ is
regular, and thus $\xi$ too is regular
(see  Theorem \ref{T: I am-stable iff} and the paragraph preceding
it), against the hypothesis.
\item[(ii)] Choose an increasing sequence of indices $n_k$ so that
$\sum_{n_k}^\infty \frac{log^k n}{n^{1+\frac{1}{k}}} \leq \frac{1}{k^2}$ and
define $\xi_n := \frac{1}{n^{1+\frac{1}{k}}}$ for $n\in[n_k,\,n_{k+1})$.
Then for every $p\in \mathbb N$,
\[
\sum_{n=n_p}^{\infty}\xi_n \log^pn
\le \sum_{k=p}^{\infty}\sum_{n=n_k}^{n_{k+1}-1}\xi_n \log^kn
\le  \sum_{k=p}^{\infty}\frac{1}{k^2} < \infty
\]
and thus $\sum_{n=1}^{\infty}\xi_n \log^pn < \infty$. 
By \cite[Proposition 4.18(ii)]{vKgW04-Traces} $(\xi) \subset st_{a_\infty}((\omega))$.
Furthermore, $\xi \neq O(\omega^p)$ for every $p >1$, so by Corollary \ref{C:inclusions}(ii), there is no $J \in \mathscr {S_{\infty}PL}$ contained in $I$.\linebreak
\end{proof}

\noindent To complete our analysis of the lattices $\Delta_{1/2}\mathscr {PL}$, $\mathscr {SPL}$, and $\mathscr {S_{\infty}PL}$ we consider also their 
complements in $\mathscr {PL}$, i.e., the principal ideals that do not satisfy the $\Delta_{1/2}$ condition, (resp., are not am-stable, or are not am-$\infty$ stable.)

\begin{proposition} \label{P:lower irregular}
Every nested pair of nonzero principal ideals has strictly between them an ideal from 
$\mathscr {PL} \setminus \Delta_{1/2} \mathscr{PL}$, $\mathscr {PL} \setminus \mathscr {SPL}$, and $\mathscr {PL} \setminus \mathscr {S_{\infty}PL}$.
\end{proposition}
\begin{proof}

Let $I \subsetneq J$ be principal ideals.
By the lack of gaps in $\mathscr {PL}$, we can find principal ideals
$I'$ and $J'$ such that
$I \subsetneq I' \subsetneq J' \subsetneq J$,
so for all three cases it is enough to find principal ideals  $I
\subset L\subset J$ outside the three special lattices considered
  without having to prove that they do not coincide with $I$ or $J$. Notice also
that since the arithmetic mean of a $\co*$ sequence always satisfies
the $\Delta_{1/2}$-condition,
  $\mathscr {SPL} \subset \Delta_{1/2} \mathscr {PL}$, and hence the
second case follows from the first. Let $I=(\xi)$ and $J=(\eta)$
and by passing if necessary to other generators, assume without loss
of generality that $\xi \leq \eta$ and $\xi_1=\eta_1$.

First we prove that there is an $I \subset L\subset J$ with $L\in
\mathscr {PL} \setminus \Delta_{1/2} \mathscr {PL}$.
Since $\eta \neq O(\xi)$, there is an increasing sequence of positive
integers  $n_k$ for which
$\eta_{n_k} \geq k\xi_{n_k}$ and $\eta_{n_k} \leq \xi_{n_{k-1}}$. \linebreak
Define
$\zeta_j := \min (\xi_{n_k}, \eta_j)$ for $j\in (n_k, n_{k+1}]$. Then
$\zeta \in \co*$,
  $\xi \leq \zeta \leq \eta$, and one has \linebreak
$\zeta_{n_k} = \eta_{n_k}\geq k\xi_{n_k} \geq k\zeta_{n_k+1} \geq
k\zeta_{2n_k}$. Thus
$L=(\zeta) \not\in \Delta_{1/2} \mathscr {PL}$ and $I\subset L\subset J$.

Now we prove that there is an $I \subset L\subset J$ with $L\in
\mathscr {PL} \setminus \mathscr {S_{\infty}PL}$.
Since $\eta \neq O(D_m\xi)$ for any $m \in \mathbb{N}$, there is an
increasing sequence of positive integers $n_k$, starting with
$n_1=1$,
such that $\eta_{n_k} \geq k^2(D_k\xi)_{n_k} = k^2\xi_{\lceil
\frac{n_k}{k} \rceil}$,
$\eta_{n_k} \leq \frac{1}{k}\eta_{n_{k-1}}$ and $n_k > kn_{k-1}$.
Define
$\zeta_i := \min (\frac{1}{k}\eta_{n_k}, \eta_i)$ for
$ i\in [\lceil \frac{n_k}{k} \rceil,\lceil \frac{n_{k+1}}{k+1} \rceil)$.
Then  $\zeta \in \co*$,
$\zeta \leq \eta$, and since for all
$ i\in [\lceil \frac{n_k}{k} \rceil,\lceil \frac{n_{k+1}}{k+1} \rceil)$,
$\frac{1}{k}\eta_{n_k} \geq k\xi_{\lceil \frac{n_k}{k} \rceil}\geq\xi_i$ and
$\eta_i \geq \xi_i$, it follows  that $\zeta \geq \xi$ and hence
$I\subset L\subset J$. Let
$q_k := \max\{i\in \mathbb N\,|\,\eta_i \geq \frac{1}{k}\eta_{n_k}\}$.
Then $q_k \geq n_k >k(\lceil \frac{n_k}{k} \rceil-1)$ and
\[
(\zeta_{a_{\infty}})_{\lceil \frac{n_k}{k} \rceil}
\geq \frac{1}{\lceil \frac{n_k}{k} \rceil}\sum_{\lceil \frac{n_k}{k}
\rceil+1}^{q_k} \frac{1}{k}\eta_{n_k}
= \frac{q_k-\lceil \frac{n_k}{k} \rceil}{\lceil \frac{n_k}{k}
\rceil}\zeta_{\lceil \frac{n_k}{k} \rceil}
\geq (k-2)\zeta_{\lceil \frac{n_k}{k} \rceil}
\]
Thus $\zeta_{a_{\infty}} \neq O(\zeta)$ and hence $L=(\zeta) \not\in
\mathscr {S_{\infty}PL}$.
\end{proof}

\noindent From this we see that above and below every nonzero
principal ideal (above only for $F$) lies a
  ``bad" ideal. Indeed, by taking $I=F$ in Proposition \ref {P:lower irregular},
we see that every principal ideal $J \neq F$ contains (properly)
principal ideals distinct from $F$ not
in the lattices $\Delta_{1/2} \mathscr {PL}$,  $\mathscr {SPL}$, and
$\mathscr {S_{\infty}PL}$. Similarly,
every principal ideal is properly contained in another principal
ideal, and hence in principal
  ideals not in the lattices $\Delta_{1/2} \mathscr {PL}$, $\mathscr
{SPL}$, and $\mathscr {S_{\infty}PL}$.

\section{\leftline {Unions of principal ideals} }\label{S:4}

Basic questions on intersections and unions of ideals from
a certain lattice play a natural role in the subject.
In \cite{nS74} Salinas investigated intersections and unions of
ideals related to various classes of mainly Banach ideals.
In this section, first we use Corollary \ref{C:inclusions} and
density Theorems \ref{T: Delta1/2-PL strong density and strong gaps},
\ref{T: SPL density in PL} and \ref{T: SinftyPL density}
to determine intersections and unions of various classes of principal ideals.
Then we investigate questions on representing ideals as unions of
chains of principal or countably generated ideals.\\

An immediate consequence of Corollary \ref{C:inclusions}(ii) is:

\begin{corollary}\label{L:intersections} \quad\\
(i) $\underset{p >0} \bigcap\,(\omega^p) = \underset{L \in \Delta_{1/2} \mathscr {PL}}\bigcap L$ 
\qquad \qquad (ii) $\underset{0 < p < 1} \bigcap\,(\omega^p) = \underset{L \in \mathscr {SPL}}\bigcap L$ 
\qquad \qquad (iii) $\underset{p > 1} \bigcup\,(\omega^p) =  \underset{L \in \mathscr {S_{\infty}PL}}\bigcup L$
\end{corollary}

\pagebreak

\begin{remark}\label{R: inters}
\item (i) By Corollaries \ref{C: count gen and inters}(ii) and \ref{C:inclusions}(i) and Theorems \ref{T: Delta1/2-PL strong density and strong gaps},
\ref{T: SPL density in PL}, and \ref{T: SinftyPL density}, every ideal in the lattices $\mathscr {PL}$, $\Delta_{1/2} \mathscr {PL}$,
$\mathscr {SPL}$, and $\mathscr {S_{\infty}PL}$ is the intersection of all ideals in the lattice that properly contain it. 
(In contrast, a principal ideal is never the union of any chain of ideals properly contained in it.)
As a consequence, every ideal in one of these lattices is the intersection of an infinite maximal chain of ideals that properly contain it and belong to the same lattice.
By \cite[Propositions 4.6 and 5.3]{vKgW04-Soft}, the chain must be
uncountable because principal ideals are ssc, while intersections of countable chains of principal ideals are never ssc.
More generally, the latter clause shows that no maximal Banach ideal 
(which are the symmetric normed ideals $\mathfrak S_\phi$ for a symmetric norming function $\phi$), 
which is ssc by \cite[Proposition 4.7]{vKgW04-Soft}, can be the proper intersection of a countable chain of principal ideals.
This provides an additional special case answer to Salinas' question in \cite[Section 7: ($\alpha_1$)]{nS74} asking whether 
$\mathfrak S_\phi \ne \bigcap I_k$ if $\mathfrak S_\phi \subsetneq I_k \subsetneq I_{k-1}$ for all $k$ (with $I_k$ not necessarily principal).
\item (ii) By \cite[Lemmas 2.1 and 2.2 and Section 3]{vKgW04-Soft},
arbitrary intersections
or directed unions of am-stable ideals (resp., am-$\infty$ stable ideals)
are am-stable (resp., am-$\infty$ stable).
\end{remark}
In the case of Banach ideals or their powers we can obtain more.
The proof used in Corollary \ref{C:inclusions}(i) to construct a
regular sequence $\zeta$
  majorizing a given sequence $\eta$ depends on the completeness of
the $\co*$-norm;
a similar proof can be employed for constructing directly a sequence
$\zeta$ that majorizes $\xi$ and satisfies the
$\Delta_{1/2}$-condition. Both constructions can be extended to
Banach ideals and the
latter even to powers of Banach ideals (i.e., $e$-complete ideals as by
  \cite[Sections 2.9-2.14 and Theorem 3.6]{DFWW}).

\begin{proposition}\label{P:complete}
\item(i) Let $I=J^p$ for some $p>0$ and some Banach ideal $J$.
Then $I = \bigcup \{L \in \Delta_{1/2}\mathscr {PL} \mid L \subset I\}$.

\item(ii) If $I$ is a Banach ideal, then $I$ is am-stable if and only if
$I = \bigcup \{L \in \mathscr {SPL} \mid L \subset I\}$.

\item(iii) If $I$ is a Banach ideal, 
then $I$ is am-$\infty$ stable if and only if $I = \bigcup \{L \in \mathscr {S_{\infty}PL} \mid L \subset I\}$.
\end{proposition}
\begin{proof}
\item(i) Clearly,  $I = \bigcup \{L \in \mathscr {PL} \mid L \subset
I\}$, thus proving (i) is equivalent to
proving that for every $\xi \in \Sigma (I)$ there is a $\zeta \geq \xi$ in
  $\Sigma (I)$ that satisfies the $\Delta_{1/2}$-condition or,
equivalently, such that $\zeta^{\frac{1}{p}}$ satisfies the
$\Delta_{1/2}$-condition. Thus without loss of generality
we can assume that $p=1$ and, from \cite[Section 4.5]{DFWW}, that $||\cdot||$ is a complete symmetric
norm on $I$. 
Denote still by  $||\cdot||$ the induced complete cone norm on $\Sigma(I)$.
Then define $\zeta :=  \sum_{m=1}^{\infty}\frac{1}{m^2||D_m \xi||} D_m\xi$.
The series converges in norm and hence $\zeta \in \Sigma (I)$. Then $\xi \leq  \zeta$.  
By the inequality for symmetric norms, $||D_2\eta|| \leq 2||\eta||$, and since $D_2 D_m=D_{2m}$, 
 $\zeta$ satisfies the $\Delta_{1/2}$-condition because

\[
 D_2\zeta =  \sum_{m=1}^{\infty}\frac{1}{m^2||D_m \xi||}D_2 D_m\xi =
\sum_{m=1}^{\infty}4\frac{|| D_{2m}\xi||}{|| D_m\xi|| }\frac{1}{(2m)^2|| D_{2m}\xi||}D_{2m}\xi
 \le 8\zeta.
 \]
\item(ii) Since the finite sum of am-stable principal ideals is am-stable,
and since the union $\bigcup\{L\in \mathscr {SPL} \mid L \subset I\}$ is
directed,  the condition is sufficient by Remark \ref{R: inters}(ii).

The proof of necessity requires showing that for every $\xi \in \Sigma (I)$,
there is a regular $\zeta \geq \xi$ in $\Sigma (I)$.
As in the proof of Corollary \ref{C:inclusions}(i), define $\zeta :=\sum_{m=1}^{\infty}\frac{1}{2^m||\xi_{a^m}||} \xi_{a^m}$. 
By the am-stability of $I$, $\xi_{a^m}\in \Sigma (I)$ for all $m$. 
Then by the completeness of the norm, $\zeta \in \Sigma (I)$. 
By \cite [Lemma 2.11] {DFWW}, the map $\Sigma(I) \ni \eta \to \eta_a \in \Sigma(I) $ is bounded, hence $\zeta$ is regular because
\[
\zeta_a =\sum_{m=1}^{\infty} \frac{1}{2^m||\xi_{a^m}||}\xi_{a^{m+1}}
\leq 2 \sup_m \left ( \frac{||\xi_{a^{m+1}}||}{||\xi_{a^m}||} \right)
\sum_{m=1}^{\infty} \frac{1}{2^{m+1}||\xi_{a^{m+1}}||}\xi_{a^{m+1}}
\leq 2 \sup_m \left ( \frac{||\xi_{a^{m+1}}||}{||\xi_{a^m}||} \right)\zeta.
\]
\item(iii) Same proof as (ii).
\end{proof}

Notice: if an ideal $I$ is ``too small", i.e., if $I \not \supset st^a((\omega))$, then it cannot contain any am-stable ideal and, in particular,
any ideal in $\mathscr{SPL}$. The same conclusion holds for $I = st^a((\omega))$ by Example \ref{E:X}(i).
Similarly, if $I \not \supset \underset{p>0}\cup(\omega)^p$, it cannot contain an ideal in $\Delta_{1/2}\mathscr{PL}$ (Corollary \ref{C:inclusions}(ii)). 
Likewise, if an ideal $I$ is ``too large", i.e., $I \not \subset \underset{p>1}\cup (\omega)^p$, then by Corollary \ref{L:intersections}(iii) it is not the union of ideals in $\mathscr {S_{\infty}PL}$. Example \ref{E:X}(ii) shows that $st_{a_\infty}((\omega))$ is ``too large." \\

While every ideal is the union of principal ideals,
a different and natural question is: which ideals are the union of an
increasing chain of principal ideals.
Salinas \cite{nS74} and others have studied this kind of question
and nested unions play an underlying role throughout the study of ideals.
A partial answer is given by the following proposition.
\begin{proposition}\label{P: nested unions and the continuum
hypothesis} Assuming the continuum hypothesis,
the following hold.
\item[(i)] All ideals are unions of increasing chains of countably
generated ideals.
\item[(ii)] An ideal $I$ is the union of an increasing chain of
principal ideals if for every countably generated ideal
  $J\subset I$ there is a principal ideal $L$ with $J\subset L\subset I$.
\end{proposition}

\begin{proof}
Since $\Sigma(I)$ has cardinality $c$, by the continuum hypothesis,
$\Sigma(I) = \{ \xi_\gamma \mid \gamma \in [1,\Omega) \}$,
i.e., it can be indexed by the interval of ordinal numbers $[1,\Omega)$
where $\Omega$ is the first uncountable ordinal.
Denote by $\prec$ the well-order relation on $[1,\Omega)$.
\item[(i)] For each $\gamma \in [1,\Omega)$, let $I_\gamma$ be the
ideal generated by the countable collection
$\{\xi_\alpha \mid \alpha \prec \gamma \}$.
Thus $I = \bigcup_{\gamma \in [1,\Omega)} I_\gamma$ is a nested (albeit uncountable) union
of countably generated ideals.
\item[(ii)] By (i) we can represent $I = \bigcup_{\gamma \in
[1,\Omega)} I_\gamma$ as a nested union
  of countably generated ideals and define using transfinite induction
a collection
$\{ \eta_\gamma \in \Sigma(I) \mid \gamma \in [1,\Omega) \}$
 so that $I_\gamma\subset (\eta_{\gamma})$ and
$(\eta_{\gamma})\subset (\eta_{\gamma'})$
  whenever $\gamma \prec \gamma' \in [1,\Omega) $.
Indeed, assume for a given $\beta \in [1,\Omega)$ that $\{ \eta_\alpha
\mid \alpha \precneqq \beta \}$
have been chosen satisfying these two conditions.
It suffices to choose an $\eta_{\beta} \in \Sigma(I)$ so that the collection  $\{ \eta_\gamma \mid \gamma \in [1,\beta] \}$
too satisfies these two conditions.
Since $[1,\beta)$ is countable, $J=\bigcup_{\gamma \in [1,\beta)} (\eta_\gamma)$ is countably generated and hence
$J \subset (\eta')$ and $I_{\beta} \subset (\eta'')$ for some $\eta', \eta'' \in \Sigma(I)$.
Set $\eta_{\beta} := \eta'+\eta'' \in \Sigma(I)$. Then for every $\gamma \prec \beta$ it follows that
$I_{\gamma}\subset (\eta_{\gamma}) \subset J+I_{\beta} \subset (\eta_{\beta})$.
Thus $\{(\eta_\gamma)\}$ forms an uncountable chain and $I = \bigcup_{\gamma \in [1,\Omega)} (\eta_\gamma)$.
\end{proof}

The following corollary shows that Banach ideals and more generally
$e$-complete ideals,
i.e., powers of Banach ideals (see \cite[Section 4.6]{DFWW}),
satisfy the condition in part (ii) of the above proposition.

\begin{corollary}\label{C:countable} 
If $I$ is an $e$-complete ideal and $J$ is a countably generated ideal with $J\subset I$, then there is a principal ideal $L$ with $J\subset L\subset I$. 
Thus, assuming the continuum hypothesis, $I$ is the union of an increasing chain of principal ideals.
\end{corollary}
\begin{proof}
Since powers of countably generated (resp., principal) ideals are countably generated (resp., principal), 
without loss of generality assume that $I$ is a Banach ideal.
Let $\{\eta^{(k)}\}$ be a sequence of generators of $J$. 
Then $\zeta := \sum_{k=1}^{\infty}\frac{1}{k^2||\eta^{(k)}||}\eta^{(k)}\in \Sigma(I)$, 
and so setting $L := (\zeta)$ one has $(\eta^{(k)}) \subset L$ for all $k$, hence $J \subset L\subset I$. 
(This is a proof that countably generated Banach ideals and hence countably generated e-complete ideals are principal. See also \cite [Corollary 2.23]{DFWW}.)
Then by Proposition \ref{P: nested unions and the continuum hypothesis}, $I$ is the union of an increasing chain of principal ideals.
\end{proof}

\noindent Notice that neither the $e$-completeness condition in the corollary nor the
condition in part (ii) of the proposition are necessary.
Indeed, for the former, any principal ideal that is not $e$-complete (i.e., whose
generator does not satisfy the $\Delta_{1/2}$-condition
\cite[Corollary 2.23]{DFWW})
  still satisfies automatically the conclusion of the corollary.
And for the latter, any
countably generated ideal that is not principal
fails to satisfy the condition in part (ii) and yet is obviously the
union of a (countable) chain of principal ideals.

We also do not know whether or not the continuum hypothesis or
something stronger than the usual axioms of set theory
is required for Proposition \ref {P: nested unions and the continuum
hypothesis}.

\begin{remark}\label{R: Ban Ideal am-stable +}
By combining Propositions \ref{P:complete} and  \ref{P: nested unions and the continuum hypothesis},
and assuming the continuum hypothesis, if $I$ is a Banach ideal,
then $I$ is am-stable if and only if it is a union of a chain of am-stable principal ideals.
And if $I$ is a positive power of a Banach ideal, then it is the union of a chain of $\Delta_{1/2}\mathscr{PL}$ ideals.
\end{remark}

\section{\leftline {Applications to first order arithmetic mean ideals} }\label{S:5}

By the main result of \cite{DFWW}, an ideal $I$ is am-stable,
i.e., $I = I_a$, if and only if $I = [I, B(H)]$ (the latter is the
commutator space of the ideal,
i.e., the span of the commutators of elements of $I$ with bounded
operators) if and only if
$I$ does not support any nonzero trace (unitarily invariant linear
functional on $I$).

We can reformulate some of the results obtained in the previous
section in terms of traces.
By Corollary \ref{C:inclusions}(i), every principal ideal $I$ is
contained in a principal ideal $J$
  supporting no nonzero traces; so no nonzero trace on $I$ can
extend to a trace on $J$.
Conversely, \linebreak
by Corollary \ref{C:inclusions}(ii), if a principal ideal
$I$ is sufficiently large to contain $(\omega^p)$ for some
\linebreak$0<p<1$, then $(\omega^p)= [(\omega^p), B(H)] \subset [I, B(H)]$
  and hence every trace on $I$ must necessarily vanish on $(\omega^p)$.
More generally, repeatedly employing Theorem \ref{T: SPL density in PL} and
Proposition \ref{P:lower irregular},
  we see that every am-stable principal ideal
is the first ideal of an increasing (or decreasing) countable chain
of principal ideals where every odd numbered ideal
  is am-stable and hence has no nonzero trace, and every even
numbered ideal supports infinitely many nonzero traces
(its space of traces is infinite-dimensional)
(\cite[Theorem 7.5]{vKgW04-Traces}, \cite[Proposition 4.6(i)]{vKgW04-Soft}).

The am-$\infty$ case is similar, but with the difference that an
ideal $I$ is am-$\infty$ stable \linebreak
(i.e., $I=\,_{a_{\infty}}I$) if and only if $I=F+[I,B(H)]$
if and only if it supports a unique nonzero trace (up to scalar multiples). 
In this case, $I \subset \mathscr L_1$ and this trace is
of course the restriction to $I$ of the usual trace $Tr$
on the trace class $\mathscr {L}_1$ \cite[Theorem 6.6]{vKgW04-Traces}.
Equivalently, by a Hamel basis argument, an ideal contained in $\mathscr {L}_1$ is not
am-$\infty$ stable if and only if it
supports a nonzero singular trace, i.e., a trace that vanishes on $F$.
And by  Theorem \ref{T: SinftyPL density} and Proposition \ref{P:lower irregular},
every am-$\infty$ stable principal ideal is the first ideal of an increasing countable chain of principal ideals 
(or a decreasing countable chain in the case that the ideal is not $F$) 
where every odd numbered ideal is am-$\infty$ stable and hence supports no nonzero singular trace, 
and every even numbered ideal supports infinitely many singular traces
(again see \cite[Theorem 7.5]{vKgW04-Traces} and \cite[Proposition 4.6(i)]{vKgW04-Soft}).\\

First order arithmetic mean ideals are investigated in \cite{DFWW} and \cite {vKgW02}-\cite{vKgW04-Soft}
and the density theorems of Section 3 provide further information for
the principal ideal cases.
Recall that am-stability of an ideal $I$ is equivalent to $I=\,_aI$.
Thus if $I$ is am-stable, so are $I_a$ and $_aI$ and hence also the first order arithmetic mean ideals $I^o := (_aI)_a$, $I^- := \,_a(I_a)$,
$I^{oo}$ (the smallest am-open ideal containing $I$), since $I\subset I^{oo} \subset I_a$, and $I_-$
(the largest am-closed ideal contained in $I$), since $_aI\subset I_-\subset I$.
The same relations hold for the am-$\infty$ case with the exception that we only have $I \cap \mathscr{L}_1 \subset I^{-\infty} \subset I_{a_{\infty}}$.
While in general these first order arithmetic mean ideals (resp., am-$\infty$ ideals) can be am-stable without the ideal $I$ being am-stable
(e.g., see Examples \ref{E:N} and \ref{E:L}), principal ideals are more ``rigid."

\begin{theorem}\label{T: I am-stable iff}
Let $I$ be a nonzero principal ideal. Then the following are equivalent.
\item[(i)] $I$ is am-stable.
\item[(ii)] $_aI$ is nonzero and am-stable.
\item[(ii$'$)] $I^o$ is nonzero and am-stable.
\item[(ii$''$)] $I_-$ is nonzero and am-stable.
\item[(iii)] $I^{oo}$ is am-stable.
\item[(iv)] $I_a$ is am-stable.
\item[(iv$'$)] $I^-$ is am-stable.
\end{theorem}
\begin{proof}
The equivalences of (ii) and (ii$'$), and of (iv) and (iv$'$), the 5-chain inclusions \linebreak
$_aI\subset I^o \subset I\subset I^-\subset I_a$ and $_aI\subset I_-
\subset I\subset I^{oo}\subset I_a$
and, consequently, the fact that (i) implies all the other conditions,
all hold for general ideals.
(See \cite[Section 2]{vKgW04-Soft} for details.)
For principal ideals, the identity $I_-=\,_aI$ is proven in [ibid.,
Theorem 2.8] and hence (ii) is equivalent to (ii$''$).
Thus it remains to prove that each of the conditions (ii$'$), (iii)
and (iv) imply (i).
By [ibid., Lemma 2.12] $I^o$, $I^{oo}$ and $I_a$ are principal when
$I$ is principal. Assume that $\{0\} \neq  I^o \subsetneq I$ and that $I^o$ is
am-stable. Then by Theorem \ref{T: SPL density in PL},
  there is an am-stable principal ideal $J$ with
$I^o \subsetneq J \subsetneq I$. Since $J=J^o$, by the monotonicity
of the am-interior operation,
$J\subset I^o$, a contradiction.
Similarly, if $I \subsetneq I^{oo}$ and $I^{oo}$ is am-stable (resp.,
$I \subsetneq I_a$ and $I_a$ is am-stable), then
there is an am-stable principal ideal $I \subsetneq  J\subsetneq
I^{oo}$ (resp., $I \subsetneq  J\subsetneq I_a$)
  and $J$ being am-open, $I^{oo} \subset J$ (resp., $I_a \subset J$), a contradiction.
\end{proof}

\noindent Notice that for any ideal $I$ that does not contain
$(\omega)$ (resp., $\mathscr L_1$), $_aI=I^o = \{0\}$
(resp., $I_-= \{0\}$). Since $\{0\}$ is am-stable while in both cases
$I$ is not,
  the conditions in the last corollary that $I_a$, $I^o$ and $I_-$
are non-zero is essential. \\

There are two main differences in the am-$\infty$ case.
The first is that if $I \neq \{0\}$, then unlike the am-case, $_{a_{\infty}}I$ is never zero, and in lieu of the two chains of inclusions, we have
$_{a_{\infty}}I \subset I^{o\infty} \subset I$, $_{a_{\infty}}I \subset I_{-\infty} \subset I$, $I\cap \mathscr L_1 \subset I^{-\infty} \subset I_{a_{\infty}}$, and $I \cap se(\omega) \subset I^{oo\infty} \subset I_{a_{\infty}}$. \linebreak
(See \cite[Proposition 4.8(i)-(i$'$)]{vKgW04-Traces} and remarks following both \cite[Corollary 3.8 and Proposition 3.14]{vKgW04-Soft}).
The second is that if $I$ is principal,
then $I^{o\infty}$, $I^{oo\infty}$ and $I_{a_{\infty}}$ are either principal or they coincide with $se(\omega)$
\cite[Lemma 4.7]{vKgW04-Traces}, \cite[Lemma 3.9, Lemma 3.16]{vKgW04-Soft}.
As $se(\omega)$ is not am-stable, if any of the ideals $I^{o\infty}$, $I^{oo\infty}$ or $I_{a_{\infty}}$ is am-stable, then it is principal as well.
With these minor changes and the properties of arithmetic mean ideals at infinity developed in \cite[Section 3]{vKgW04-Soft},
the proof of Theorem \ref{T: I am-stable iff} carries over to the am-$\infty$ case.

\begin{theorem}\label{T: I am-inf stable iff}
Let $I \neq \{0\}$ be a principal ideal. Then am-$\infty$ stability
of the following ideals is equivalent:
$I$, $_{a_{\infty}}I$, $I^{o{\infty}}$, $I_{-{\infty}}$,
$I^{oo\infty}$, $I_{a_{\infty}}$, and $I^{-\infty}$.
\end{theorem}

The fact that a sequence $\xi \in \co* $ is regular if and only if
$\xi_a$ is regular has been proved in a
number of different ways (see \cite[Remark 3.11]{DFWW} for a discussion).
The analogous result for $\infty$-regularity was obtained in
\cite{vKgW04-Traces}.
Both facts are also an immediate consequence of Theorems \ref{T: I
am-stable iff} and \ref{T: I am-inf stable iff}:

\begin{corollary}\label{C: higher regularities} Let $I$ be a
principal ideal and let $n\in \mathbb N$.
\item[(i)]$I_{a^n}$ is am-stable if and only if $I$ is am-stable.
\item[(ii)]$I_{a_\infty^n}$ is am-$\infty$ stable if and only if $I$
is am-$\infty$ stable.
\end{corollary}

It may be of interest to note that Theorem \ref{T: I am-stable iff}
does not extend beyond principal ideals. In fact, as the two examples
below show, it can fail even for countably generated ideals.
It is an open question whether or not Theorem \ref{T: I am-inf stable iff}
extends beyond principal ideals, indeed,
even whether or not am-$\infty$ stability of $I_{a_{\infty}}$
implies am-$\infty$ stability of $I$.

\begin{example}\label{E:N}
A countably generated ideal $N$ which is not am-stable but for which \linebreak
$_aN =N^o =N_-= st^a((\omega))$ is am-stable.
\end{example}

\begin{proof}
Define the sequence $\zeta_i = 2^{-k^3}$ for $i\in [2^{k^3}, 2^{(k+1)^3})$.
Since
\[\zeta_{2^{(k+1)^3-1}}= 2^{-k^3} \neq  O((\omega log^k)_{2^{(k+1)^3-1}})\]
and since $\omega log^k$ is increasing in $k$, $\zeta \neq  O(\omega log^p)$
for every $p\in \mathbb N$.
As $\omega log^p$ satisfies the $\Delta_{1/2}$-condition,
  it follows that  $\zeta \not\in \Sigma((\omega log^p))$  and hence
$\zeta \not\in \Sigma(st^a((\omega)))$.
On the other hand, if for some $p\in \mathbb N$ and some $\rho \in \co*$,
$\rho_a \leq \zeta + \omega log^p$, then for all $i\in [2^{k^3},
2^{(k+1)^3})$ and for $k$ large enough,
\begin{align*}
i(\rho_a)_i &\leq 2^{(k+1)^3} (\rho_a)_{2^{(k+1)^3}} \leq 2^{(k+1)^3}
\left (2^{-(k+1)^3}
+ \frac{log^p (2^{(k+1)^3})}{2^{(k+1)^3}}\right)\\
&= 1 + log^p (2^{(k+1)^3}) \leq 2 log^p (2^{k^3}) \leq 2log^p i = 2i
(\omega log^p)_i.
  \end{align*}
Thus $\rho_a = O(\omega log^p)$, which proves that $(\zeta + \omega log^p)^o = (\omega log^p)$ since $(\omega log^p) = (\omega)_{a^p}$ which is am-open.
Let $N:= (\zeta)+st^a((\omega))= \bigcup_p (\zeta + \omega log^p)$.
Then $N$ is countably generated and $N\neq st^a((\omega))$.
 By \cite[Lemma 2.1]{vKgW04-Soft}, $N^o =  \bigcup_p (\zeta + \omega log^p)^o =  \bigcup_p (\omega log^p)= st^a((\omega))$. Thus $N^o=\,_a(N^o) =\,_aN \ne N$. 
By [ibid, Theorem 2.8], $N_-=\,_aN $, which concludes the proof.
\end{proof}

\begin{example}\label{E:L}
A countably generated ideal $L$ which is not am-stable but for which \linebreak
$L_a =L^- =L^{oo} = st^a((\omega))$ is am-stable.
\end{example}

\begin{proof}
We construct inductively a sequence of sequences $\xi^{(p)} \in
\text{c}_{o}^*$  for which the principal ideals
$(\xi^{(1)}+\cdots+\xi^{(p)}) \not  \supset (\omega)$  and
$(\xi^{(p)})_a = (\omega)_{a^{p+1}}$. The initial $p=1$ case is a special case of the general induction.
Assume the construction up to $p -1\ge 1$. Since $\omega \neq O(\xi^{(1)}+\cdots+\xi^{(p-1)})$,
choose an increasing sequence of integers $r_i$ for which
$(\frac{\xi^{(1)}+\cdots+\xi^{(p-1)}}{\omega})_{r_i} \to 0$.
For the initial case $p=1$ simply choose $r_i=i$.

Construct inductively two strictly increasing sequences of integers,
$q_{k-1} < n_{k-1} < q_k$ with $\frac{n_k}{log^pn_k} >(k+1)2^p$,
  starting with $q_o = 1$ and with an integer $n_o$
chosen so that $\frac{n_o}{log^pn_o} >2^p$
and $n_o > e^p$ (the latter condition ensures that $\frac{log^p j}{j}$ is decreasing for $j \ge n_o$).
Assuming the construction of $q_j,n_j$ up to $k-1 \geq 0$, choose
$q_k := \max\,\{j\in \mathbb N \mid \frac{1}{kn_{k-1}} \leq \frac{log^p j}{j}\}$; thus  $kn_{k-1} < q_k < n_{k-1}^2$.
  Then choose $n_k> q_k+1$ from among the sequence
$\{r_i\}$ and sufficiently large to satisfy the conditions
$\sum_{j=q_k+1}^{n_k-1} \frac{log^p j}{j} \geq \frac{1}{2}\sum_{j=1}^{n_k} \frac{log^p j}{j}$
and $\frac{n_k}{log^pn_k} >(k+1)2^p$.

Now define the sequence $\xi^{(p)}_j := \min\,\{\frac{1}{kn_{k-1}},\frac{log^p j}{j}\}$ for
$n_{k-1} \leq j < n_k$  and $k \ge 1$, setting $\xi^{(p)}_j :=1$  for $1 \le j < n_o$.
Then $\xi^{(p)} \in \co*$, $\xi^{(p)}_j \leq (\omega log^p)_j$ for $j\ge n_o$,
and so $(\xi^{(p)})\subset (\omega log^p)$.
An elementary computation shows that the sequence inequalities
\[
C_p\,\omega\, log^{p+1} \leq (\omega log^p)_a \leq \omega log^{p+1}
\]
hold for some constant $0 < C_p < 1$, hence the ideal identities
$(\omega log^p)_a = (\omega log^{p+1})=(\omega)_{a^{p+1}}$ also hold. Thus
$(\xi^{(p)})_a \subset (\omega)_{a^{p+1}}$.
To obtain equality, we show that $(\omega log^p)_a = (\xi^{(p)})_a$.
Indeed, for $k \ge 2$ and  $n_{k-1} \leq j \leq q_k$,
\begin{align*}
j(\xi^{(p)}_a)_j &\geq \sum_{i=q_{k-1}+1}^{n_{k-1}-1}\xi_i^{(p)}
= \sum_{i=q_{k-1}+1}^{n_{k-1}-1}\frac{log^p i}{i}
\ge \frac{1}{2} \sum_{i=1}^{n_{k-1}}\frac{log^p i}{i}
\ge \frac{C_p}{2} log^{p+1}n_{k-1}\\
&\ge\frac{C_p}{2} log^{p+1} \sqrt{q_k}
\ge  \frac {C_p}{2^{p+2}}log^{p+1}j
\ge   \frac {C_p}{2^{p+2}}j((\omega log^p)_a)_j.
\end{align*}
\noindent For $k \ge 2$ and $q_k < j < n_k$, by the above inequality applied at $q_k$ we also have
\begin{align*}
j(\xi^{(p)}_a)_j &= q_k(\xi^{(p)}_a)_{q_k} + \sum_{i=q_k+1}^j \frac{log^p i}{i}
\geq  \frac {C_p}{2^{p+2}}q_k((\omega log^p)_a)_{q_k} +\sum_{i=q_k+1}^j \frac{log^p i}{i}\\
&\ge \frac {C_p}{2^{p+2}} \left(\sum_{i=1}^{q_k} \frac{log^pi}{i} +
\sum_{i=q_k+1}^j \frac{log^p i}{i} \right)
  = \frac {C_p}{2^{p+2}}j((\omega log^p)_a)_j.
\end{align*}
Thus $(\omega log^p)_a \le \frac{2^{p+2}}{{C_p}} \xi^{(p)}_a$  on
$[n_1, \infty)$, which proves
that $(\omega)_{a^{p+1}} = (\xi^{(p)})_a$.
Finally,
\[
(\frac{\xi^{(1)}+\cdots+\xi^{(p)}}{\omega})_{n_k}
= (\frac{\xi^{(1)}+\cdots+\xi^{(p-1)}}{\omega})_{n_k}
+ (\frac{\xi^{(p)}}{\omega})_{n_k} =
(\frac{\xi^{(1)}+\cdots+\xi^{(p-1)}}{\omega})_{n_k} + \frac{1}{k+1} \to 0.
\]
Thus $\omega \ne O((\xi^{(1)}+\cdots+\xi^{(p)}))$ and since $\omega$
satisfies the $\Delta_{1/2}$-condition,
it follows that
$(\omega) \not\subset (\xi^{(1)}+\cdots+\xi^{(p)})$.
This completes the inductive construction of the sequence of
sequences $\xi^{(p)}$.

Now let $L$ be the ideal generated by the sequences $\xi^{(p)}$.
Since $L$ is the union of the principal ideals
$(\xi^{(1)}+\cdots+\xi^{(p)})$, by construction $(\omega) \not\subset
L$.
Since $L_a$ is the ideal generated by $\{\xi^{(p)}_a \mid p\in
\mathbb{N}\}$, and since
  $(\xi^{(p)})_a = (\omega_{a^{p+1}})$,
$L_a$ coincides with $st^a((\omega)) = \bigcup_{p=0}^\infty
(\omega)_{a^p} $, the upper am-stabilizer  of the principal ideal $(\omega)$
(see Section 2).
Thus $L_a$ is am-stable and so $L_a =\, _a(L_a) = L^-$.
Also $L_a$ contains $(\omega)$ while $L$ does not, so $L_a \neq L$, i.e., $L$ is not
am-stable.

To obtain equality with $L^{oo}$, we first show that $(\xi^{(p)})^{oo} = (\omega)_{a^p}$ for every $p$.
Because $(\xi^{(p)})\subset (\omega log^p) = (\omega)_{a^p}$ which is am-open, by definition $(\xi^{(p)})^{oo}\subset (\omega log^p)$.
By \cite[Remark 2.19]{vKgW04-Soft}, $(\xi^{(p)})^{oo}$ is principal.
Choose as a generator of it some $\rho_a \geq  \xi^{(p)}$ (see \cite[Lemma 2.13]{vKgW04-Soft}).
For  every $k\ge 1$ and  $j \in (q_k, n_k)$ one has $j(\rho_a)_j \geq j\xi^{(p)}_j = log^pj= j(\omega log^p)_j $ from the definition of the sequence $\xi^{(p)}$.
For every $j\in [n_k, q_{k+1}]$, by using the last inequality at $n_k-1$,
\begin{align*}
j(\rho_a)_j &\geq (n_k-1)(\rho_a)_{(n_k-1)}
 \geq log^p(n_k-1)
 \geq log^p(\sqrt{q_{k+1}}-1)\\
&\geq log^p(\sqrt{j}-1)
\geq log^p(j^{1/4}) = 4^{-p} j(\omega log^p)_j.
\end{align*}
The last inequality holds because $j \ge n_k \ge (\frac{1+\sqrt 5}{2})^4$.
Thus $\omega log^p = O(\rho_a)$ and hence $(\xi^{(p)})^{oo} = (\omega log^p)$.
By \cite[Lemma 2.22(i) and (v)]{vKgW04-Soft},
\begin{align*}
L^{oo} &= \bigl (\bigcup_p \left
((\xi^{(1)})+\cdots+(\xi^{(p)})\right )\bigr)^{oo}
= \bigcup_p \left ((\xi^{(1)})+\cdots+(\xi^{(p)})\right )^{oo}\\
&=  \bigcup_p \left ((\xi^{(1)})^{oo}+\cdots+(\xi^{(p)})^{oo}\right )
=\bigcup_p (\omega log^p) = st^a((\omega)).
\end{align*}
\end{proof}

\section{\leftline {First order cancellation properties}} \label{S:6}
\noindent In this section we use sublattice density to study first order arithmetic mean ideals' cancellation properties for inclusions.
Several of these results were announced in \cite{vKgW02}.
We start with the elementary equivalences.

\begin{lemma}\label{L:closed open} Let $I$ be an ideal.
\item[(C)]  The following conditions are equivalent.\\
\indent (i) $I$ is am-closed \\
\indent (ii) $J_a\subset I_a \Rightarrow J\subset I$
\indent (ii\,$'$) $J^-\subset I^-\Rightarrow J\subset I$\\
\indent (iii) $J_a = I_a\Rightarrow J\subset  I$
\indent (iii\,$'$) $J^-= I^-\Rightarrow J\subset  I$\\
\indent (iv) $I_-\subset J_-\Rightarrow I\subset J$\\
\indent (v) $I_-= J_-\Rightarrow I\subset  J$

\item[(O)]  The following conditions are equivalent.\\
\indent (i) $I$ is am-open \\
\indent (ii) $_aI\subset \,_aJ\Rightarrow I\subset J$
\indent (ii\,$'$) $I^o\subset \,J^o\Rightarrow I\subset J$ \\
\indent (iii) $_aI = \,_aJ\Rightarrow I\subset J$
\indent (iii\,$'$) $I^o=J^o\Rightarrow I\subset  J$ \\
\indent (iv) $J^{oo}\subset  I^{oo}\Rightarrow J\subset I$\\
\indent (v) $J^{oo}= I^{oo}\Rightarrow J\subset  I$

\item[(C$_{\infty}$)]  The following conditions are equivalent.\\
\indent (i) $I$ is am-$\infty$ closed \\
\indent (ii) $J \subset \mathscr{L}_1$ and $J_{a_{\infty}}\subset
I_{a_{\infty}}\Rightarrow J\subset I \subset \mathscr{L}_1$
\indent (ii$'$) $J \subset \mathscr{L}_1$ and $J^{-{\infty}}\subset
I^{-{\infty}}\Rightarrow J\subset I \subset \mathscr{L}_1$ \\
\indent (iii) $J \subset \mathscr{L}_1$ and $J_{a_\infty}=
I_{a_\infty}\Rightarrow J\subset I \subset \mathscr{L}_1$
\indent (iii$'$) $J\subset \mathscr{L}_1$ and $J^{-{\infty}}=
I^{-{\infty}}\Rightarrow J\subset I \subset \mathscr{L}_1$\\
\indent (iv) $I_{-{\infty}}\subset
J_{-{\infty}}\Rightarrow I\subset J$ \\
\indent (v) $I_{-{\infty}}=
J_{-{\infty}}\Rightarrow I\subset J$

\item[(O$_{\infty}$)] The following conditions are equivalent.\\
\indent (i) $I$ is am-$\infty$ open \\
\indent (ii) $_{a_{\infty}}I\subset\, _{a_{\infty}}J\Rightarrow I\subset J$
\indent (ii$'$) $I^{o\infty} \subset J^{o\infty}\Rightarrow I \subset  J$ \\
\indent (iii) $_{a_{\infty}}I=\, _{a_{\infty}}J\Rightarrow I\subset J$
\indent (iii$'$) $I^{o\infty} = J^{o\infty}\Rightarrow I \subset  J$ \\
\indent (iv)  $J \subset se (\omega)$ and $J^{oo\infty}\subset I^{oo\infty}\Rightarrow J\subset I\subset se (\omega)$\\
\indent (v)  $J \subset se (\omega)$ and $J^{oo\infty}= I^{oo\infty}\Rightarrow J\subset  I \subset se (\omega)$

\end{lemma}
\begin{proof}
Conditions (x) and (x$'$) are all equivalent by the monotonicity (i.e., inclusion preserving property) of the operations 
$L \mapsto L_a$, $L \mapsto\, _aL$, $L \mapsto L_{a_{\infty}}$, and  $L \mapsto \,_{a_{\infty}}L$ 
and by the indentities $L^-= (\,_a(L_a))_a$, $L^o= \,_a((_aL)_a)$,  $L^{-{\infty}}= (\,_{a_{\infty}}(L_{a_{\infty}}))_{a_{\infty}}$, 
and $L^{o{\infty}}= \,_{a_{\infty}}((_{a_{\infty}}L)_{a_{\infty}})$ (see Section \ref{S:2}). 
Also, condition (ii) always trivially implies condition (iii) and condition (iv) always trivially implies condition (v). 
The am-case being similar but simpler because of the 5-chains of inclusions, we prove only the am-$\infty$ case.
\item [(C$_{\infty}$)] Recall that $_{a_{\infty}}I \subset I_{-{\infty}}\subset I$ and $I \cap \mathscr{L}_1 \subset I^{-{\infty}}\subset I_{a_{\infty}}\cap \mathscr{L}_1$ for every ideal $L$ (see paragraph preceding Theorem \ref{T: I am-inf stable iff}).\\
(i) $\Rightarrow$ (ii$'$) $I = I^{-{\infty}} \subset \mathscr {L}_1$ and $J = J\cap \mathscr{L}_1 \subset J^{-{\infty}} \subset  I^{-{\infty}} = I$.\\
(iii$'$) $ \Rightarrow $ (i) Since $I^{-{\infty}} = (I^{-{\infty}})^{-{\infty}}$  and $I^{-{\infty}} \subset \mathscr{L}_1$, it follows that $I^{-{\infty}} \subset I \subset \mathscr{L}_1$. But then  $I = I \cap \mathscr{L}_1 = I^{-{\infty}}$ and hence $I$ is am-$\infty$ closed. \\
(i)$\Rightarrow$ (iv) $I= I_{-{\infty}}\subset J_{-{\infty}}\subset J$.\\
(v)$\Rightarrow$ (i) Since $I_{-{\infty}}= (I_{-{\infty}})_{-{\infty}}$ and $I_{-{\infty}}\subset \mathscr{L}_1$, it follows that $I\subset I_{-{\infty}}$ and hence $I= I_{-{\infty}}$, i.e., $I$ is am-$\infty$ closed.
\item [(O$_{\infty}$)]  Recall that $_{a_{\infty}}L \subset L^{o\infty} \subset L$ and 
$L \cap se(\omega) \subset L^{oo\infty} \subset L_{a_{\infty}}\subset se (\omega)$ for every ideal $L$ (see ibid).\\
(i) $\Rightarrow$ (ii$'$) $I= I^{o\infty} \subset J^{o\infty}\subset J$.\\
(iii$'$) $ \Rightarrow $ (i) Since $I^{o\infty} =(I^{o\infty})^{o\infty} $, it follows that $I \subset I^{o\infty} $ so  $I = I^{o\infty} $, i.e., $I$ is am-$\infty$ open. \\
(i)$ \Rightarrow $ (iv) As every am-$\infty$ open ideal, $I \subset se (\omega)$. 
Moreover, $J= J \cap se (\omega) \subset J^{oo\infty}\subset  I^{oo\infty} = I$.
(v)$ \Rightarrow $ (i) Since $I^{oo\infty}=(I^{oo\infty})^{oo\infty}$ and $I^{oo\infty}\subset se(\omega)$, 
it follows that $I^{oo\infty} \subset  I \subset se(\omega)$, hence $I=I \cap  se (\omega) = I^{oo\infty} $, i.e., $I$ is am-$\infty$ open. 
\end{proof}

A reformulation of (C) in the above lemma is that $I$ is not am-closed if and only if there is an ideal $L\not\subset I$ for which $I_a=L_a$, 
in which case $I$ is contained in the strictly larger ideal $J:= I+L$ with the same arithmetic mean as $I$.
In case $I$ is countably generated (resp., principal) the next proposition shows that we can choose that larger ideal $J$ to also be countably generated 
(resp., principal). The same holds for the am-$\infty$ case.

\begin{proposition}\label{P:count gen}
\item[(i)] If $I$ is a countably generated ideal that is not
am-stable, then there is a countably
generated ideal $J \supsetneq I$ such that $J_a = I_a$. If $I$ is principal,
then $J$ can be chosen to be principal.
\item[(ii)] If $I$ is a countably generated ideal that is not
am-$\infty$ stable, then there is a countably
generated ideal $J \supsetneq I$ such that $J_{a_\infty} = I_{a_\infty}$.
If $I$ is principal, then $J$ can be chosen to be principal.
\end{proposition}
\begin{proof}
By \cite[Corollaries 2.9 and 3.5]{vKgW04-Soft}, if $I$ is countably
generated and not am-stable (resp., am-$\infty$ stable,)
then it is not am-closed (resp., am-$\infty$ closed).
By the upper density of $\mathscr {L}_{\aleph_o}$ and
$\mathscr {PL}$ in $\mathscr L$, there is a $J \in \mathscr
{L}_{\aleph_o}$  such that $I\subsetneq J\subsetneq I^-$
(resp.,  $I\subsetneq J\subsetneq I^{-_{\infty}}$), and if $I\in
\mathscr{PL}$ then $J$ can be chosen to be principal.
But then $I_a\subset J_a\subset (I^-)_a=I_a$ (resp.,
$I_{a_{\infty}}\subset J_{a_{\infty}}\subset
(I^-)_{a_{\infty}}=I_{a_{\infty}}$).
\end{proof}
The answer is different for the reverse inclusion question: If $I$ is not
am-stable, can we always find a smaller ideal $J \subsetneq I$ such
that $J_a=I_a$?
The next example shows that the answer is negative even when $I$
is principal and that the same holds for the am-$\infty$ case.

\begin{example}\label{E:lower KD}
\item[(i)]
Let $\xi_i:=\frac{1}{k!}$ for $((k-1)!)^2 <i\leq (k!)^2$ and all
$k\geq 2$ and let $\xi_1=1$.
Then the principal ideal $(\xi)$ is not am-stable and $J\subsetneq
(\xi)$ implies $J_a \neq (\xi)_a$.
\item[(ii)] Let $\xi_i :=\frac{1}{(k!)^2}$ for $(k-1)! < i \leq k!$
and all $k\geq 2$ and let $\xi_1=1$.
Then the principal ideal $(\xi)$ is not am-$\infty$ stable
and $J\subsetneq (\xi)$ implies $J_{a_{\infty}}\neq (\xi)_{a_{\infty}}$.
\end{example}
\begin{proof}
\item[(i)]
It is easy to verify that $\xi \in \co*$ and that $\xi$ does not
satisfy the $\Delta_{1/2}$-condition
  and hence is not am-stable.
Now assume by contradiction that there is an ideal $J\subsetneq (\xi)$
with $J_a = (\xi)_a$.
Then $(\xi)_a = (\eta)_a$ for some $(\eta)\subset J$, that is, $J$
can also be chosen to be principal.
Without loss of generality, since $\xi$ and every $MD_m\xi$
for $M>0$ and $m\in \mathbb N$, generate the same ideal,
we can assume that $\eta \le \xi$. Since $(\eta)\ne(\xi)$, $\xi \ne O(D_j\eta)$ for every $j\in \mathbb N$, i.e.,
$\xi_{r_j} \ge j \eta_{\lceil \frac{r_j}{j}\rceil}$ for some strictly
increasing sequence of integers $r_j$.
Furthermore, we can assume that the intervals $((k_j-1)!)^2 < r_j \leq (k_j!)^2$ containing distinct $r_j$
are disjoint, that is, $k_j$ is strictly increasing.
Then $\eta_{\frac{(k_j!)^2}{j}} \leq
\eta_{\lceil \frac{r_j}{j}\rceil}\leq \frac{1}{j}\xi_{r_j}
=\frac{1}{j(k_j!)}$ and since $\sum_1^k n! \le 2k!$,
\begin{align*}
(k_j&!)^2(\eta_a)_{(k_j!)^2} = \sum_{i=1}^{(k_j!)^2}\eta_i
\leq
\sum_{i=1}^{\frac{(k_j!)^2}{j}}\xi_i+\sum_{i=\frac{(k_j!)^2}{j}+1}^{(k_j!)^2}\eta_{\frac{(k_j!)^2}{j}}\\
&= 1+\sum_{n=2}^{k_j-1}\frac{(n!)^2-((n-1)!)^2}{n!}
+\left (\frac{(k_j!)^2}{j}-((k_j-1)!)^2 \right)\frac{1}{k_j!}
+\left ( (k_j!)^2-\frac{(k_j!)^2}{j}\right )\eta_{\frac{(k_j!)^2}{j}}\\
&\leq 2(k_j-1)! + \frac{k_j!}{j}+ \frac{k_j!}{j}\leq 4
\frac{k_j!}{j}= \frac{4}{j}(k_j!)^2\xi_{(k_j!)^2}
\leq \frac{4}{j}(k_j!)^2(\xi_a)_{(k_j!)^2}.
\end{align*}
But then $\xi_a \neq O(\eta_a)$, and hence $(\eta_a)\neq (\xi_a)$, a
contradiction.
\item[(ii)] Since $\sum_{k=1}^{\infty}\frac{1}{(k!)^2}(k!-(k-1)!)<\infty$, by \cite[Example 4.5(iii), Corollary 4.10 and Definition 4.11]{vKgW04-Traces}, $(\xi)$ is not am-$\infty$ stable.
As in the proof of part (i), reasoning by contradiction we can assume
that $(\eta)\subsetneq (\xi)$ but
$(\eta_{a_{\infty}}) = (\xi_{a_{\infty}})$ for some $\eta \in (\ell^1)^*$
with $\eta \le \xi$, and we can choose
a sequence $(k_j-1)! < r_j \le k_j!$ with strictly increasing $k_j$ for which
$\xi_{r_j} \ge j \eta_{\lceil \frac{r_j}{j}\rceil}$.
Then $j\eta_{\frac{k_j!}{j}} \le  j \eta_{\lceil \frac{r_j}{j}\rceil} \le \xi_{r_j}
=\xi_{k_j!}=\frac{1}{(k_j!)^2} $ and since $\sum _{k}^\infty \frac{1}{n!} \leq \frac{e}{k!}$,
\begin{align*}
{k_{j^2}!}(\eta_{a_{\infty}})_{\frac{k_{j^2}!}{j^2}}
&=j^2\biggl( \sum _{i=\frac{k_{j^2}!}{j^2}+1}^{k_{j^2}!}\eta_i +
\sum _{i=k_{j^2}!+1}^{\infty}\eta_i\biggr)
\le j^2\biggl(k_{j^2}! \eta_\frac{k_{j^2}!}{j^2} + \sum
_{n=k_{j^2}+1}^{\infty}\frac{n!-(n-1)!}{n!^2}\biggr)\\
&\le j^2\biggl(\frac{1}{j^2k_{j^2}!} + \sum _{n=k_{j^2}+1}^{\infty}\frac{1}{n!}\biggr)
\le j^2\left(\frac{1}{j^2k_{j^2}!}+ \frac{e}{(k_{j^2}+1)!}\right)
  \le \frac{4}{k_{j^2}!}.
\end{align*}
On the other hand, for $j \ge 2$,
\begin{align*}
{k_{j^2}!}(\xi_{a_{\infty}})_{\frac{k_{j^2}!}{j}}
&\ge j \sum _{i=\frac{k_{j^2}!}{j}+1}^{k_{j^2}!}\xi_i
= j (k_{j^2}! - \frac{k_{j^2}!}{j})\frac{1}{(k_{j^2}!)^2}
  \ge \frac{j}{2k_{j^2}!}.
\end{align*}
Therefore $(\xi_{a_{\infty}})_{\frac{k_{j^2}!}{j}} \ge \frac{j}{8}(\eta_{a_{\infty}})_{\frac{k_{j^2}!}{j}}$ and hence
$\xi_{a_{\infty}}\ne O(D_m\eta_{a_{\infty}})$ for any
integer $m$. Thus
$(\eta_{a_{\infty}}) \neq (\xi_{a_{\infty}})$.
As in part (i), this concludes the proof.
\end{proof}

\begin{remark}\label{R:1} That $(\eta)_a \supset (\xi)$ implies
$(\eta)_a \supset (\xi)_a$
for the sequence $\xi$ of part (i) of the above example
(or equivalently, that $(\xi)^{oo}=(\xi)_a$)
was proved in  \cite [Example 2.20]{vKgW04-Soft}.
This combined with Example \ref{E:lower KD}(i) shows that
if $(\eta) \subset (\xi)\subset (\eta)_a$, then $(\eta) = (\xi)$.
That is, $(\xi)$ cannot lie strictly between
any principal ideal and its arithmetic mean.
\end{remark}
At an early stage of this project, Ken Davidson and the second named
author found, for the case of principal ideals,
a direct constructive proof of Proposition \ref{P:count gen}(i) or rather,
equivalently, that there is a principal ideal
$J\supsetneq I$ with $J^-=I^-$. 
The same result was obtained earlier by Allen and Shen \cite{gAlS78} by different methods. 
The proof that we presented here is a
consequence of the identity $I_-=\,_aI$ for countably generated ideals
  (see \cite [Theorem 2.9]{vKgW04-Soft}). \\

To handle the am-interior
(resp., am-$\infty$ interior) analogs of Proposition \ref{P:count gen} for principal ideals, a direct constructive proof seems needed. 
This we present in the next proposition.
\begin{proposition}\label{P:open'}
\item[(i)]
Let  $(\rho)\ne\{0\}$ be a principal ideal that is not am-stable.
Then there are two principal ideals $(\eta^{(1)})$ and $(\eta^{(2)})$
with $(\eta^{(1)})$ possibly zero, such that
\[
(\eta^{(1)}) \subsetneq (\rho) \subsetneq (\eta^{(2)}),\quad
(\eta^{(1)})^{oo} = (\rho)^o, \quad \text{and} \quad
(\rho)^{oo}=(\eta^{(2)})^o.
\]
\item[(ii)] Let  $(\rho)\ne\{0\}$ be a principal ideal that is not am-$\infty$ stable.
Then there are two principal ideals $(\eta^{(1)})$ and $(\eta^{(2)})$ such that
\[
(\eta^{(1)}) \subsetneq (\rho) \subsetneq (\eta^{(2)}), \quad
(\eta^{(1)})^{oo\infty} = (\rho)^{o\infty}, \quad \text{and} \quad
(\rho)^{oo\infty}=(\eta^{(2)})^{o\infty}.
\]
\end{proposition}
\begin{proof}
\item[(i)]
By \cite[Lemma 2.14]{vKgW04-Soft}, $(\rho)^o$ and  $(\rho)^{oo}$  are principal. Also  $(\rho)^o=\{0\}$ precisely when $(\omega) \not \subset (\rho)$.
If  $(\rho)$ is not am-open and hence  $(\rho)^o \subsetneq (\rho)\subsetneq(\rho)^{oo}$,
then choosing $(\eta^{(1)}) :=  (\rho)^o$ and $(\eta^{(2)}) := (\rho)^{oo}$ satisfies the requirement.
Thus assume that $(\rho)$ is am-open and hence $(\rho)^o = (\rho)^{oo}  = (\rho)$.
By \cite[Lemma 2.13]{vKgW04-Soft} $\rho$ is equivalent to $\xi_a$ for some  $\xi \in \co*$.
Thus  $\xi_a$ is irregular and by \cite[Remark 3.11]{DFWW} (see also Corollary \ref{C: higher regularities}(i)), $\xi$ too is irregular, i.e., $\xi_a \ne O(\xi)$.
Choose an increasing sequence of indices  $n_k$ such that
$(\xi_a)_{n_k} \geq k\xi_{n_k}$ and without loss of
  generality assume that $n_{k+1} > kn_k$. Then
\[
kn_k(\xi_a)_{kn_k} =  \sum_{j=1}^{n_k} \xi_j + \sum_{j=n_k+1}^{kn_k} \xi_j
\leq    n_k(\xi_a)_{n_k}  + kn_k\xi_{n_k} \leq 2n_k(\xi_a)_{n_k}
\]
and hence $(\xi_a)_{n_k} \geq \frac{k}{2}(\xi_a)_{kn_k}$. \\
Define
$\eta_j^{(1)} =
\begin{cases}
(\xi_a)_{kn_k}    &\text{for $n_k \leq j \leq kn_k$} \\
(\xi_a)_j    &\text{for $kn_k \leq j < n_{k+1}$}
\end{cases}$ and
$\eta_j^{(2)} =
\begin{cases}
(\xi_a)_{n_k}    &\text{for $n_k \leq j \leq kn_k$} \\
(\xi_a)_j    &\text{for $kn_k < j \leq n_{k+1}$}.
\end{cases}$\\
Then $\eta^{(i)} \in \co*$  for $i = 1, 2$ and $\eta^{(1)}  \leq
\xi_a \leq  \eta^{(2)}$.
Thus $(\eta^{(1)}) \subset (\xi_a) \subset (\eta^{(2)})$ and since
$(\xi_a)$ is am-open,
$(\eta^{(1)})^{oo} \subset (\xi_a) \subset (\eta^{(2)})^o$.
We need to prove strict containments in the first pair of inclusions
and equalities in the second pair.
Since
\[
\left(\frac{\xi_a}{\eta^{(1)}}\right)_{n_k} = \left(\frac{\eta^{(2)}}{\xi_a}\right)_{kn_k}
= \frac{(\xi_a)_{n_k}}{(\xi_a)_{kn_k}} \geq \frac{k}{2},
\]
i.e.,
$\xi_a \ne O(\eta^{(1)})$ and $\eta^{(2)} \ne O(\xi_a)$, and since
$\xi_a$ satisfies the $\Delta_{1/2}$-condition,
  it follows that $(\eta^{(1)}) \ne (\xi_a) \ne (\eta^{(2)})$.

By \cite[Lemma 2.14]{vKgW04-Soft}, $(\eta^{(1)})^{oo}$ is principal,
i.e., $(\eta^{(1)})^{oo} = (\mu_a)$ for some $\mu \in \co*$ [ibid,
Lemma 2.13].
Since  $(\eta^{(1)})\subset (\mu_a)$ and $\mu_a$ satisfies the
$\Delta_{1/2}$-condition,
by multiplying $\mu$ by a constant if necessary, one can assume
without loss of generality that $\mu_a \geq \eta^{(1)}$.
For all $n_k \leq j \leq kn_k$,
\begin{align*}
j(\mu_a)_j     &\geq      (n_k - 1)(\mu_a)_{n_k-1}
         \geq     (n_k - 1)\eta^{(1)}_{n_k-1}
         =    (n_k - 1)(\xi_a)_{n_k-1} \\
          &\geq     (n_k - 1)(\xi_a)_{n_k}         \geq
\frac{k}{2}(n_k - 1)(\xi_a)_{kn_k}
         \geq    \frac{1}{4} kn_k(\xi_a)_{kn_k}
          \geq    \frac{1}{4} j(\xi_a)_j.
\end{align*}
For all $kn_k \leq j < n_{k+1}$ we have $j(\mu_a)_j  \geq
j\eta^{(1)}_j = j(\xi_a)_j$.
Thus $\xi_a=O(\mu_a)$ and hence $(\xi_a)\subset (\mu_a)
= (\eta^{(1)})^{oo}$, whence $(\xi_a) = (\eta^{(1)})^{oo}$.

Similarly there is a  $\nu_a \leq \eta^{(2)}$ such that $(\nu_a) =
(\eta^{(2)})^o$.
Then for all $n_k \leq j \leq kn_k$,
\begin{align*}
j(\nu_a)j      &\leq      (kn_k+1)(\nu_a)_{kn_k+1}         \leq
(kn_k+1)\eta^{(2)}_{kn_k+1}
          =     (kn_k+1)(\xi_a)_{kn_k+1} \\
&\leq     2kn_k(\xi_a)_{kn_k}          \leq     4n_k(\xi_a)_{n_k}
\leq     4 j(\xi_a)_j.
\end{align*}
For all $kn_k < j \leq n_{k+1}$ we have $j(\nu_a)j   \leq
j\eta^{(2)}_j  =  j(\xi_a)_j$.
Thus $\nu_a  = O(\xi_a)$ and hence $(\nu_a)  \subset (\xi_a)$, whence
$(\eta^{(2)})^o =(\xi_a)$.
\item[(ii)] We consider separately the cases when $(\rho)$ is and when it is
not am-$\infty$ open.

First assume that
$(\rho)$ is not am-$\infty$ open and hence
$(\rho)^{o\infty}\subsetneq (\rho) \ne (\rho)^{oo\infty}$.
Recall that by \cite[Lemma 3.9(i)]{vKgW04-Soft}, $(\rho)^{o\infty}$ is not
principal precisely when $(\omega) \subset  (\rho)$,
in which case \linebreak
$(\rho)^{o\infty}= se(\omega)$.
By \cite[Lemmas 3.9(ii) and 3.16(ii)]{vKgW04-Soft},
$(\rho)^{oo\infty}$ is not principal precisely when
$(\rho)\not\subset se(\omega)$, in which case $(\rho)^{oo\infty}= se(\omega)$.
Hence if $(\rho) \subset se(\omega)$, then both $(\rho)^{o\infty}$ and
$(\rho)^{oo\infty}$
are principal and $(\rho)^{o\infty}\subsetneq (\rho)\subsetneq
(\rho)^{oo\infty}$, so it suffices to choose
  $(\eta^{(1)}) = (\rho)^{o\infty}$ and $(\eta^{(2)}) = (\rho)^{oo\infty}$.
And otherwise if $(\rho) \not\subset se(\omega)$, then choose a sequence
$\eta^{(2)}$ sufficiently large so that
  $(\rho+\omega) \subsetneq (\eta^{(2)})$.
Then $(\eta^{(2)})^{o\infty} = se(\omega) = (\rho)^{oo\infty}$.
To obtain $\eta^{(1)}$ in the case when $(\omega) \subset  (\rho)$,
choose $0\neq\eta^{(1)} \leq \omega$ but $\eta^{(1)} \ne o(\omega)$. Then $\omega \ne O(\eta^{(1)})$
hence $(\eta^{(1)}) \subsetneq (\omega)$, and $(\eta^{(1)}) \not\subset se(\omega)$
(which follows since $\omega$ satisfies the $\Delta_{1/2}$-condition).
Then $(\eta^{(1)}) \subsetneq (\rho)$ and $(\eta^{(1)})^{oo\infty} =
se(\omega) =  (\rho)^{o\infty}$.
In the case when $(\omega) \not \subset  (\rho)$, choose
$(\eta^{(1)}) = (\rho)^{o\infty}$ since then the latter ideal is
principal.

Thus, as in part (i), it remains to consider only the case that
$(\rho)$ is am-$\infty$ open
  and hence $(\rho) = (\xi_{a_\infty})$ for some $\xi \in \ell_1^*$
(see \cite[Lemma 3.7]{vKgW04-Soft}).
By \cite[Theorem 4.12(i) and (iv)]{vKgW04-Traces}, since $(\rho)$ is
not am-$\infty$ stable, $\xi$ is am-$\infty$ irregular, and
again by [ibid, Theorem 4.12(i) and (ii)], $\xi_{a_\infty} \ne O(\xi)$.
Thus we can find an increasing sequence of indices  $n_k$ such that
$(\xi_{a_\infty})_{n_k} \geq 2k\xi_{n_k}$,
and without loss of generality we can assume that $n_{k+1} > kn_k$.
Then
\[
n_k(\xi_{a_\infty})_{n_k}    =  \sum_{j=n_k+1}^{kn_k} \xi_j   +
    \sum_{j=kn_k+1}^{\infty} \xi_j  \leq    kn_k\xi_{n_k} +
kn_k(\xi_{a_\infty})_{kn_k}
             \leq   \frac{1}{2} n_k(\xi_{a_\infty})_{n_k}+ kn_k(\xi_{a_\infty})_{kn_k}
\]
and hence $(\xi_{a_\infty})_{n_k} \leq  2k(\xi_{a_\infty})_{kn_k}$.
As for the am-case, define \\
\[
\eta_j^{(1)} =
\begin{cases}
(\xi_{a_\infty})_{kn_k}    &\text{for $n_k \leq j \leq kn_k$} \\
(\xi_{a_\infty})_j    &\text{for $kn_k \leq j < n_{k+1}$}
\end{cases} \quad \text{and} \quad
\eta_j^{(2)} =
\begin{cases}
(\xi_{a_\infty})_{n_k}    &\text{for $n_k \leq j \leq kn_k$} \\
(\xi_{a_\infty})_j    &\text{for $kn_k < j \leq n_{k+1}$}.
\end{cases}
\]
\noindent Then $\eta^{(i)} \in \co*$  for $i = 1, 2$ and
$\eta^{(1)} \leq  \xi_{a_\infty} \leq  \eta^{(2)}$. Thus $(\eta^{(1)}) \subset (\xi_{a_\infty}) \subset (\eta^{(2)})$.
Moreover, since $ \xi_{a_\infty} \in \Sigma(se(\omega))$ it follows that
$(\eta^{(1)}) \subset se(\omega)$, and since $(\frac{\omega}{\eta^{(2)}})_{n_k} = (\frac{\omega}{\xi_{a_\infty}})_{n_k}\rightarrow \infty$,
$(\eta^{(2)}) \not \supset (\omega)$.
Thus  both
$(\eta^{(1)})^{oo\infty}$ and $ (\eta^{(2)})^{o\infty}$ are principal and
$(\eta^{(1)})^{oo\infty} \subset (\xi_{a_\infty}) \subset
(\eta^{(2)})^{o\infty}$ \cite[Lemma 3.9(i)-(ii)]{vKgW04-Soft}.
For $m \in \mathbb N$ and all $k \geq m$ we have
\[
\left(\frac{\xi_{a_\infty}}{D_m\eta^{(1)}}\right)_{kn_{k^2}}
\geq \left(\frac{\xi_{a_\infty}}{D_k\eta^{(1)}}\right)_{kn_{k^2}}
= \frac{\left(\xi_{a_\infty}\right)_{kn_{k^2}}}{\eta^{(1)}_{n_{k^2}}}
=
\frac{\left(\xi_{a_\infty}\right)_{kn_{k^2}}}{\left(\xi_{a_\infty}\right)_{k^2n_{k^2}}}
\geq k
\]
since, as is easy to verify, $(\zeta_{a_\infty})_m \geq
k(\zeta_{a_\infty})_{km}$ for any positive integers
$k, m$ and $\zeta \in (\ell^1)^*$.
Thus $(\xi_{a_\infty}) \ne (\eta^{(1)})$.  Similarly, for all $k \geq m$,
\[
\left(\frac{\eta^{(2)}}{D_m\xi_{a_\infty}}\right)_{k^2n_{k^2}}
\geq \left(\frac{\eta^{(2)}}{D_k\xi_{a_\infty}}\right)_{k^2n_{k^2}}
=
\frac{\left(\xi_{a_\infty}\right)_{n_{k^2}}}{\left(\xi_{a_\infty}\right)_{kn_{k^2}}}
\geq k,
\]
whence $(\eta^{(2)})\ne (\xi_{a_\infty})$. Since  $(\eta^{(1)})\subset
(\eta^{(1)})^{oo\infty}$
   and the latter is am-$\infty$ open and principal,
by \cite[Lemma 4.1(i)]{vKgW04-Traces} we can choose one of its generators
$\mu_{a_\infty}$ so that $\mu_{a_\infty} \geq \eta^{(1)}$.
Then for all $n_k \leq j \leq kn_k$, using $(\xi_{a_\infty})_{n_k}
\leq  2k(\xi_{a_\infty})_{kn_k}$,
\[
j(\mu_{a_\infty})_j     \geq    kn_k(\mu_{a_\infty})_{kn_k}
         \geq    kn_k\eta^{(1)}_{kn_k}
         =    kn_k(\xi_{a_\infty})_{kn_k}
         \geq   \frac{1}{2} n_k(\xi_{a_\infty})_{n_k}
         \geq    \frac{1}{2}  j(\xi_{a_\infty})_j,
\]
and for all $kn_k \leq j < n_{k+1}$ we have $j(\mu_{a_\infty})_j  \geq
j(\eta^{(1)})_j
= j(\xi_{a_\infty})_j$. Thus $\xi_{a_\infty} = O(\mu_{a_\infty} )$, hence $(\xi_{a_\infty}) \subset (\mu_{a_\infty}) = (\eta^{(1)})^{oo\infty}$, whence
$(\xi_{a_\infty}) = (\eta^{(1)})^{oo\infty}$.

Similarly, by the proof of \cite[Lemma 3.7]{vKgW04-Soft}, there is a generator $\nu_{a_\infty}$ for $(\eta^{(2)})^{o\infty}$
such that $\nu_{a_\infty} \le \eta^{(2)}$.
If $n_k \leq j \leq kn_k$, then
\[
     j(\nu_{a_\infty})_j     \leq    n_k( \nu_{a_\infty})_{n_k}
         \leq    n_k\eta^{(2)}_{n_k}
         =    n_k(\xi_{a_\infty})_{n_k}
         \leq    2kn_k(\xi_{a_\infty})_{kn_k}         \leq 2j(\xi_{a_\infty})_j.
\]
If $kn_k < j \leq n_{k+1}$, then $j(\nu_{a_\infty})_j  \leq
j\eta^{(2)}_j = j(\xi_{a_\infty})_j$.
Thus $\nu_{a_{\infty}}  = O(\xi_{a_{\infty}})$, hence $(\eta^{(2)})^{o\infty} = (\nu_{a_{\infty}}) \subset (\xi_{a_{\infty}})$, 
whence $(\eta^{(2)})^{o{\infty}} = (\xi_{a_{\infty}})$.
\end{proof}

\begin{remark}\label{R:subset existence from Varga}
\item[(i)] The existence of the sequence of indices $n_k$ constructed
in the above proof
can also be derived for case (i) from \cite[Lemma 1]{jV89}
  and for case (ii) from \cite[Theorem 4.12(v)]{vKgW04-Traces}.
\item[(ii)] In Proposition \ref{P:open'} the condition of the non
am-stability of the principal ideal $(\rho)$  is necessary.
Indeed if $(\rho)$ were am-stable and hence am-open,
$(\eta^{(1)})^{oo} = (\rho)$ or $(\eta^{(2)})^{o}=(\rho)$
would imply by Theorem \ref{T: I am-stable iff} that
$(\eta^{(1)})=(\eta^{(1)})^{oo}= (\rho)$ or
  $(\eta^{(2)})=(\eta^{(1)})^{o}= (\rho)$. The same holds for the
am-$\infty$ case with the same argument.
\end{remark}

\begin{theorem}\label{T: cancellation}
\item[(A)] Let $I\ne \{0\}$ be a principal ideal and let $J$ be an
arbitrary ideal. Then the following are equivalent
\item[(i)] $I$ is am-stable
\item[(ii)] $_aJ = \,_aI$  (or, equivalently, $ J^o = I^o$) implies $J = I$
\item[(iii)] $ J_a = I_a$  (or, equivalently, $ J^- = I^-$) implies $J = I$
\item[(iv)] $ J^{oo} = I^{oo}$  implies $J = I$
\item[(v)] $ J_- = I_-$  implies $J = I$

\item[(B)] Let $I\ne \{0\}$ be a principal ideal and let $J$ be an
arbitrary ideal. Then the following are equivalent
\item[(i)] $I$ is am-$\infty$ stable
\item[(ii)] $_{a_\infty}J =\, _{a_\infty}I$ (or, equivalently, $
J^{o\infty} = I^{o\infty}$) implies $J = I$.
\item[(iii)] $J_{a_\infty} =I_{a_\infty}$  (or, equivalently, $
J^{-\infty} = I^{-\infty}$) implies $J = I$
\item[(iv)] $ J^{oo\infty} = I^{oo\infty}$  implies $J = I$
\item[(v)] $J_{-\infty} = I_{-\infty}$  implies $J = I$
\end{theorem}

\begin{proof}
As the am-case is similar and slightly simpler than the am-$\infty$ case, we prove only the latter. Recall from \cite [Corollary 4.10]{vKgW04-Traces} that $I$ is am-$\infty$ stable if and only if $I=\, _{a_\infty}I$, or, equivalently, if and only if $I\subset \mathscr L_1$ and $I = \,I_{a_\infty}$, in which case also
\[
I= I_{-\infty}=I^{-\infty}= I^{o\infty}=I^{oo\infty}\subset st_{a_{\infty}}(\mathscr{L}_1) \subsetneq \mathscr{L}_1.
\]
\noindent (i) $\Rightarrow $ (ii)-(v) Since $I$ is am-$\infty$ open (resp. am-$\infty$ closed),  if the hypothesis of (ii) (resp., (v)) are satisfied, then in either case, $I\subset J$ by  Lemma \ref {L:closed open}.  If the hypothesis of (iv) (resp., (iii)) are satisfied, then $J^{oo\infty} \ne se (\omega)$  (resp., $J^{-\infty} \ne \mathscr{L}_1$), hence by \cite [Lemma 3.16 (ii) ]{vKgW04-Soft}, $J\subset se(\omega)$   (resp., by  \cite [ Lemma 3.16 (i) ]{vKgW04-Soft}, $J\subset \mathscr{L}_1$ and thus in either case, $J\subset I$ by Lemma \ref {L:closed open}.

Now assume by contradiction that $I\neq J$.
Then in cases (ii) and (v), by the upper density of $\mathscr {S_{\infty} PL}$ in $\mathscr {L}$ (see Corollary \ref{C: count gen upper density}(ii) and Theorem \ref{T: SinftyPL density}) there is an ideal $L \in \mathscr {S_{\infty} PL}$ strictly between $I$ and $J$.
In cases (iii) (resp., (iv)), there is a principal ideal $(\eta) \subset J$ for which $(\eta)_{a_{\infty}}=I_{a_{\infty}}=I$ (resp., $(\eta)^{oo\infty} = I^{oo\infty}=I$).
Since $I\neq (\eta)$ by the assumption $I\neq J$,
by Theorem \ref{T: SinftyPL density} there is again an ideal $L \in \mathscr {SPL}$ strictly between $I$ and $(\eta)$.
Since $L=\,_{a_{\infty}}L=L_{-{\infty}}=L^{oo{\infty}}=L_{a_{\infty}}$, this yields an immediate contradiction for all four cases.
\noindent (ii),(iv) $\Rightarrow $ (i) Since $(I^{o{\infty}})^{o\infty} = I^{o\infty}$ (resp.,
$(I^{oo{\infty}})^{oo{\infty}} = I^{oo{\infty}}$), it
follows that $I$ is am-${\infty}$ open. The conclusion is now immediate from
Proposition \ref{P:open'}(ii).

\noindent (iii),(v) $\Rightarrow $ (i) If $I$ is not am-$\infty$ stable, then
it is not am-$\infty$  closed \cite[Theorem 3.5]{vKgW04-Soft} and so
it cannot coincide with $I^{-\infty}$ (resp., with  $I_{-\infty}$), but $(I^{-\infty})^{-\infty}
=I^{-\infty}$ (resp.,  $(I_{-\infty})_{-\infty}=I_{-\infty}$), contradicting the hypothesis.
\end{proof}

\noindent  Examples \ref{E:N} and \ref{E:L} show that the first order
equality cancellations of
Proposition \ref{T: cancellation} can fail for countably generated
am-stable ideals.
Indeed, $_aN=N_-$ is am-stable yet \linebreak
$_aN=\,_{a2}N \not \Rightarrow N=\,_aN$ and
$N_-=\,(N_-)_- \not \Rightarrow N=\,N_-$. Similarly,
$L_a = L^{oo}$ is am-stable yet $L_a=L_{a2}$ does not imply $L=L_a$ and
$ L^{oo}=( L^{oo}) ^{oo}$ does not imply $L=L^{oo}$.\\

First order cancellations involving inclusions are less simple even for principal ideals. \linebreak
As shown in Lemma \ref{L:closed open}, necessary and sufficient conditions on the ideal $I$ for the cancellations $_aI\subset\, _aJ \Rightarrow I\subset J$ (resp., $J_a\subset I_a \Rightarrow J\subset I$) are straighforward: $I$ must be am-open  (resp., am-closed).
The ``opposite'' implications, $_aJ\subset \,_aI \Rightarrow J\subset I$ and $I_a\subset J_a \Rightarrow I\subset J$ do not 
hold in general even when $I$ is principal and am-stable.
For the latter,  Corollaries \ref{C: 5} and \ref {C: 6} show that this cancellation fails for every $I = (\omega^p)$ ($0 < p < 1$) and Example \ref{E:7} provides a principal ideal where this cancellation does hold.
For the former, the cancellation fails for every principal ideal $I$, as shown by the next proposition which actually proves more.

\begin{proposition}\label{P:open} Let $I$ be an ideal. Then for every principal
  ideal $(\xi)$ there is a principal ideal $(\zeta) \not\subset (\xi)$
with $_a(\zeta) \subset \,_aI$.
\end{proposition}
\begin{proof}
Assume first that $_aI \neq {0}$ and hence $\mathscr L_1 \subset \,_aI$
(by \cite[Lemma 6.2(iii)-(iv)]{vKgW04-Soft}).
Let $n_k$ be an increasing sequence of integers such that $\xi_ {\lceil\frac{n_{k+1}-1}{k}\rceil} \leq \frac{1}{kn_k}$.
Define the sequence $\zeta_j := \frac{1}{n_k}$ for $j\in[n_k, n_{k+1})$.
Then $\zeta \in \co*$ and since $\zeta_{n_{k+1}-1} \geq k(D_k\xi)_{n_{k+1}-1}$ for every $k$, it follows that $(\zeta) \not\subset (\xi)$.
Now let $\rho \in \Sigma (_a(\zeta))$, that is,  $\rho_a \in  \Sigma ((\zeta))$, and since $\rho_a$ satisfies the $\Delta_{1/2}$-condition,
assume without loss of generality that $\rho_a \leq \zeta$.
Then $\sum_{i=1}^{n_k}\rho_i = n_k(\rho_a)_{n_k} \leq 1$ and hence $\rho \in \ell^1$.
Thus $_a(\zeta) \subset \mathscr L_1 \subset \, _aI$.

For the case when $_aI =\{0\}$, let  $n_k$ be a strictly
increasing sequence of integers such that
$\xi_ {\lceil\frac{n_{k+1}-1}{k}\rceil} \leq \frac{1}{k^2n_k}$.
Define the sequence
$\zeta_j := \frac{1}{kn_k}$ for $j\in[n_k, n_{k+1})$. Then it is easy
to verify that $(\zeta) \not\subset (\xi)$ and
$(\omega) \not \subset (\zeta)$, and hence $_a(\zeta) = \{0\}$.
\end{proof}

As a consequence of Proposition \ref{P:open}, recalling from \cite[Theorem 2.9] {vKgW04-Soft} that $(\zeta)_- = \,_a(\zeta)$ for every principal ideal $(\zeta)$,
we see that the cancellation $J_-\subset I_-\Rightarrow J \subset I$ also fails for every principal ideal $I$.

The proof of Proposition \ref{P:open} can be adapted to show that
while principal ideals have principal am-interiors,
the converse does not hold.

\begin{example}\label{E:principal interior} There is a non-principal
ideal with principal (nonzero) am-interior.
\begin{proof}
Define the sequences $\zeta^{(p)}_j = \frac{1}{(2^pk)!}$ for $j\in
[(2^pk)!,\, (2^p(k+1))!)$ for every $p\in \mathbb N$.
  Then it is easy to verify that
  $(\omega)\subset (\zeta^{(p)})\subsetneq(\zeta^{(p+1)})$ for all
$p$, from which it follows that the ideal $I:= \bigcup_p (\zeta^{(p)})$
  is not principal.
On the other hand, if $\{0\}\ne (\rho_a) \subset (\zeta^{(p)})$ is an
 am-open ideal, and without loss of generality, $\rho_a \leq \zeta^{(p)}$, then the inequality $(2^pk)!(\rho_a)_{(2^pk)!} \leq 1$ shows that $\rho \in \ell^1$. 
Thus, as in the proof of the above proposition, $(\rho_a) = (\omega)$, that is, $( \zeta^{(p)})^{o}= (\omega)$.
By \cite[Theorem 2.17]{vKgW04-Soft}, $I^o = \bigcup_p (\zeta^{(p)})^o = (\omega)$, which is a principal ideal.
\end{proof}
\end{example}

Proposition  \ref{P:open} implies that for any given ideal $I$, there
is never a ``least upper bound'' principal ideal for the solutions $J$ of the inclusion $_aJ \subset\, _aI$. 
Intuitively, a candidate might be the union of such ideals but this, in general, is not an ideal.
In contrast however, the intersection of ideals being always an ideal, for every given ideal $I$, the ``greatest lower bound," $\widehat I$,
for the solutions of the inclusion  $J_a\supset I_a $,  i.e.,  the largest ideal for which \linebreak
$J_a\supset I_a \Rightarrow J\supset \widehat I$, is given by:

\begin{definition}\label{D:Gg} 
$\widehat{I}:=\bigcap\{J \mid J_a \supset I_a\}$
\end{definition}

\noindent We will identify $\widehat I $ for principal ideals and show that it is itself principal (when $I$ is principal) and that it may be strictly smaller than $I$ (even when $I$ is am-stable), though not always (see Corollary \ref{C: 6}(ii), Example \ref{E:7} and Proposition \ref{P:Gg is stable}). 

The motivation for considering $\widehat I $ comes from the still open question from \cite[Section 7]{DFWW}:
If the class $[I,B(H)]_1$ of single commutators of operators in $I$ with operators in $B(H)$,
contains a finite rank operator with nonzero trace, must $\Sigma(I)$ contain $\sqrt \omega$?
It was shown in \cite[Theorem 7.3]{DFWW} that $\Sigma(I^-) \ni \sqrt \omega$,
i.e., there is an $\eta \in \Sigma(I)$ with $\eta_a \geq  (\sqrt \omega)_a \asymp \sqrt \omega$.
However, we will show in Corollary \ref{C: 6}(ii) (for $p=\frac{1}{2}$) that the latter condition does not imply $(\eta) \supset (\sqrt \omega)$
but only that $(\eta) \supset (\omega)$ and that this ``lower bound" $(\omega)$ is sharp for this implication.
For general $\xi$ it can be shown that there is never a minimal $\widehat{\xi}$ 
for which $\eta_a\ge \xi_a$ implies $\eta \ge \widehat{\xi}$, 
but as Theorem \ref{T: 4}(i) shows there is a minimal one asymptotically.

\begin{definition}\label{D: 1}
For $\xi \in \text{c}_\text{o}^* \setminus \ell^1$,
let $\nu(\xi)_{n} := \min \{k \in \mathbb N \mid
\sum_{i=1}^{k}\xi_{i} \geq n\xi_1 \}$ and define
$\widehat{\xi} :=\, <(\xi_{a})_{\nu(\xi)_{n}} >$.
\end{definition}

\begin{lemma}\label{L: 2}
If $\xi \in \text{c}_\text{o}^* \setminus \ell^1$,
then $\widehat{\xi} \in \text{c}_\text{o}^*$ and $(1 -
\frac{1}{\nu(\xi)_{n}})\, \widehat{\xi}_{n}
< \frac{n}{\nu(\xi)_{n}}\xi_1 \leq \widehat{\xi}_{n}$.
\end{lemma}

\begin{proof}
Since $\xi \notin \ell^1$, the distribution sequence $\nu(\xi)$ is
well-defined, nondecreasing, and tends to infinity since
$\nu(\xi)_{n} \geq n$ because
$n\xi_1 \leq \sum_{i=1}^{\nu(\xi)_n}\xi_{i} \leq \nu(\xi)_n \xi_1$.
Thus $\widehat{\xi} \in \co*$. For each fixed $m$ and all $n$ for
which $\nu(\xi)_n \geq m$,
\[n\xi_1 \leq \sum_{i=1}^{\nu(\xi)_n}\xi_{i} = \sum_{i=1}^{m-1}\xi_{i} +
  \sum_{i=m}^{\nu(\xi)_n}\xi_{i} \leq \sum_{i=1}^{m-1}\xi_{i} +
\nu(\xi)_n \xi_m\]
so
\[0 \leq \underset{n}{\overline{\lim}}\,\frac{n}{\nu(\xi)_n}\xi_1
\leq
\underset{n}{\overline{\lim}}\,\frac{\sum_{i=1}^{m-1}\xi_{i}}{\nu(\xi)_n}
+ \xi_m = \xi_m.\]
Since $m$ is arbitrary,  $\frac{n}{\nu(\xi)_n} \rightarrow 0$. Moreover,
\[
   (\nu(\xi)_{n} - 1)\, \widehat{\xi}_{n} \leq (\nu(\xi)_{n} -
1)\,(\xi_{a})_{\nu(\xi)_{n} - 1} =
  \sum_{i=1}^{\nu(\xi)_{n} - 1} \xi_{i} < n\xi_1 \leq
\sum_{i=1}^{\nu(\xi)_{n}} \xi_{i}
= \nu(\xi)_{n}\,\widehat{\xi}_{n},
\]
whence the claim.
\end{proof}

\noindent Notice that $\widehat{\xi}$ is asymptotic to $<\frac{n}{\nu(\xi)_n}\xi_1>$.
Also the following lemma is  left to the reader.

\begin{lemma}\label{L: 3} For all $\xi, \eta \in \co* \setminus
\ell^1$ and all $t >0$ we have
\item[(i)] $\nu(t\xi) =\nu(\xi)$ \quad and \quad  $\widehat{t\xi} = t\,\widehat{\xi}$
\item[(ii)] $\min \{\nu(\xi),\,\nu(\eta)\}\leq \nu(\xi+\eta) \leq \max \{\nu(\xi),\,\nu(\eta)\}$
\quad and \quad  $\widehat{\xi+\eta} \leq \widehat{\xi}+\widehat{\eta}$
\end{lemma}

\pagebreak

\begin{theorem}\label{T: 4}
Let $\xi \in \co* \setminus \ell^1$.
\item[(i)] If $\eta_a \geq \xi_a$ and $\eta \in \co*$, then for every
$\varepsilon > 0$,
$\eta_{n} \geq (1 - \varepsilon) \,\widehat{\xi}_{n}$
for $n$ sufficiently large.
\item[(ii)]
For every $\rho \in \co*\setminus \Sigma((\widehat \xi))$,
there is an $\eta \in \co*$ for which $\eta_a \geq \xi_a$ and
$\rho \notin \Sigma((\eta))$.
\end{theorem}
\begin{proof}
\item[(i)] By Lemma \ref{L: 3}(i), assume without loss of generality
that $\xi_1 = 1$. For every $n$ large enough
  so that $(\eta_{a})_{n} \leq \frac{\varepsilon}{2}$ and
$\nu(\xi)_{n} > \max(n, \frac{2}{\varepsilon})$, we have
\[
   n \leq \sum_{i=1}^{\nu(\xi)_{n}} \xi_{i} \leq
\sum_{i=1}^{\nu(\xi)_{n}} \eta_{i}
\leq \sum_{i=1}^{n} \eta_{i} + (\nu(\xi)_{n} - n)\eta_{n} \leq
\frac{\varepsilon}{2}n + (\nu(\xi)_{n} - n)\eta_{n}
\]
and thus
\[
\eta_{n} \geq (1 - \frac{\varepsilon}{2})\,\frac{n}{\nu(\xi)_{n} - n}
\geq (1 - \frac{\varepsilon}{2})\,\frac{n}{\nu(\xi)_{n}}
>    (1 - \frac{\varepsilon}{2})^{2}\, \widehat{\xi}_{n}> (1 -
\varepsilon)\, \widehat{\xi}_{n}
\]
where the third inequality follows from Lemma \ref{L: 2}.
\item[(ii)] Assume without loss of
generality that $\xi_1  = 1$ and, again by Lemma \ref{L: 3}(i),  that $ \rho_1 = 1$.
We construct inductively an increasing sequence of integers $n_k$ and three
derived sequences
  $l_k = \lceil \frac{4k(k + 1)}{\rho_{n_k}}\rceil$, $m_k:=
\lceil\frac{n_k}{k}\rceil$, and
$p_k:= \lceil\frac{n_k}{l_{k-1}}\rceil$ starting with $l_0 = n_1 =1$.
Assume the construction up to $k-1
\geq 1$. Since  $\rho \ne O(D_{l_{k-1}}\widehat \xi)$,
we can choose an integer $n_k$ large enough so that
\begin{itemize}
\item[(a)] $\rho_{n_k} \geq k (D_{l_{k-1}}\widehat \xi)_{n_k} = k(\widehat \xi)_{p_k}$
\item[(b)] $n_k \geq l_{k-1}$
\item[(c)] $\rho_{n_k} \leq \frac{k^2}{2l_{k-1}}$
\item[(d)] $m_k \geq 2m_{k-1}$
\item[(e)] $m_k > \nu(\xi)_{p_{k-1}}$
\end{itemize}
Now define the sequence $\eta_j:= \frac{1}{k}\rho_{n_k}$ for $m_k
\leq j < m_{k + 1}$.
Then $\eta \in \co*$ and $\rho \ne \Sigma((\eta))$ because
$\rho_{n_k} = k(D_k\eta)_{n_k}$ for all $k$.
  The key to prove that $\eta_a\geq \xi_a$ is the following lower
bound for $\nu(\xi)_{p_k}$:
\[
\nu(\xi)_{p_k} \geq \frac{p_k}{\widehat \xi_{p_k}} \geq
\frac{kn_k}{l_{k-1}\rho_{n_k}} \geq 2\frac{n_k}{k}> m_k,
\]
where the first inequality follows from Lemma \ref{L: 2}, the second
from (a) and the definition of ${p_k}$,
the third from (c) and the fourth from the definition of $m_k$. Thus by (e),
 $m_k <  \nu(\xi)_{p_k} < m_{k+1}$. If $m_k \leq j < \nu(\xi)_{p_k}$, then
\begin{align*}
\sum_{i=1}^{j}\eta_i - \sum_{i=1}^{j}\xi_i
&\geq \sum_{i=m_{k-1}}^{m_k-1}\eta_i - \sum_{i=1}^{\nu(\xi)_{p_k}-1}\xi_i  &\text{obvious}\\
&> (m_k-m_{k-1})\eta_{m_{k-1}}- p_k & \text{by the definition of $\nu(\xi)_{p_k}$}\\
&>  \frac {(m_k-m_{k-1})}{k-1}\rho_{n_{k-1}} - \frac{n_k}{l_{k-1}} -1
&\text{by the definitions  of $\rho$ and of $p_k$}\\
&> \frac{n_k}{2(k-1)k}\rho_{n_{k-1}} -2 \frac{n_k}{l_{k-1}} &\text{by (d) and (b) and since $m_k \ge \frac{n_k}{k}$}\\
&\geq  \frac{n_k}{2(k-1)k}\frac{4(k-1)k}{l_{k-1}} - 2\frac{n_k}{l_{k-1}} = 0
&\text{by the definition of $l_k$.}
\end{align*}
If now $\nu(\xi)_{p_k} \leq j < m_{k+1}$, then
\begin{align*}
\sum_{i=1}^{j}\eta_i - \sum_{i=1}^{j}\xi_i &=
\sum_{i=1}^{\nu(\xi)_{p_k}-1}\eta_i -
\sum_{i=1}^{\nu(\xi)_{p_k}-1}\xi_i
+ \sum_{i=\nu(\xi)_{p_k}}^{j}\eta_i - \sum_{i=\nu(\xi)_{p_k}}^{j}\xi_i\\
&\ge \sum_{i=\nu(\xi)_{p_k}}^{j}\eta_i - \sum_{i=\nu(\xi)_{p_k}}^{j}\xi_i &\text{by the above inequality}  \\
&\geq (j- \nu(\xi)_{p_k}+1)\left(\frac{1}{k}\rho_{n_k} -
\xi_{\nu(\xi)_{p_k}}\right) &\text{by the monotonicity of $\eta$ and $\xi$} \\
&\geq (j- \nu(\xi)_{p_k}+1)\left(\frac{1}{k}\rho_{n_k} -
(\xi_a)_{\nu(\xi)_{p_k}}\right)\\
&=  (j- \nu(\xi)_{p_k}+1)\left(\frac{1}{k}\rho_{n_k} -
\widehat{\xi}_{p_k}\right) \geq 0 &\text{by the definition of $\widehat{\xi}$ and (a)}.
\end{align*}
Therefore $\eta_a\geq \xi_a$, which completes the proof.
\end{proof}

Theorem \ref{T: 4} permits us to compute $\widehat{I}$ (see Definition \ref{D:Gg}) for every principal ideal $I$.

\begin{corollary}\label{C: 5}
Let $0 \ne \xi \in \co*$, then
$\widehat{(\xi)} =
\begin{cases}
F &\text{\quad for~$\xi \in \ell^1$}\\
(\widehat\xi) &\text{\quad for~$\xi\notin \ell^1$}
\end{cases}$
\end{corollary}

\begin{proof}
If $\xi \in \ell^1$, then $(\xi_a) = (\omega) = F_a$, thus $\widehat{(\xi)} = F$.
If $\xi \ne \ell^1$, then for every ideal $J$ such that $J_a \supset
(\xi)_a$, there is an
  $\eta \in \Sigma(J)$ such that $\xi_a \leq \eta_a$.
By Theorem \ref{T: 4}(i), $( \widehat{\xi} ) \subset (\eta) \subset J$, hence
$( \widehat{\xi} ) \subset  \widehat{(\xi)} $.
By Theorem \ref{T: 4}(ii), for every $\rho \in \co* \setminus
\Sigma((\widehat \xi))$, there
  is an $\eta \in \co*$ such that $(\eta)_a \supset (\xi)_a$ but $\rho
\notin \Sigma((\eta))$.
But then $\rho \notin \Sigma(\bigcap \{ J \mid J_{a} \supset
(\xi)_{a}\}) = \Sigma(\widehat{(\xi)})$,
whence the claim.
\end{proof}

\begin{corollary}\label{C: 6}
\item[(i)]
\[ \widehat{(\omega)^p} =
\begin{cases}
F &\text{for $p>1$}\\
(<\frac{n}{e^n}>) &\text{for $p=1$}\\
(\omega)^{p'} &\text{for $0<p<1$, where $\frac{1}{p} -\frac{1}{p'}=1$}
\end{cases}
\]
\item[(ii)] $\widehat{(\omega)^p}$ is am-stable if and only if
$0<p<\frac{1}{2}$.
\end{corollary}

\begin{proof}
\item[(i)] For $p>1$, $\omega^p$ is summable and hence, by Corollary \ref{C: 5}, $\widehat{(\omega)^p} =F$.
A routine computation shows $\nu(\omega)_n \asymp e^n$ and therefore, again by Corollary \ref{C: 5} and Lemma \ref{L: 2}, one has \linebreak
$\widehat{(\omega)}=(\widehat{\omega}) =(<\frac{n}{e^n}>)$. 
Similarly, for $0<p<1$, $(\nu(\omega)^p)_n \asymp n^{\frac{1}{1-p}}$, hence on has \linebreak 
$\widehat{(\omega)^p}=(<\frac{n}{n^{\frac{1}{1-p}}}>) = (\omega^{p'})$.
\item[(ii)] The ideals $F$ and $(<\frac{n}{e^n}>)$ are clearly not am-stable, and $(\omega)^{p'}$ is am-stable if and only if $0<p'<1$.
\end{proof}

So, in particular $\widehat{(\omega)^{1/2}} = (\omega)$. By
definition, $\widehat{(\xi)}\subset (\xi)$, but as the
case of $(\omega)^{1/2}$ illustrates, the inclusion can be proper
even when $\xi$ is regular, i.e., when $(\xi)$ is
am-stable. As the following proposition shows, the inclusion is
certainly proper when $(\xi) \neq F$ is not am-stable.

\begin{proposition}\label{P:Gg is stable} If $F\neq (\xi) = \widehat{(\xi)}$ with $\{0\}\neq\xi \in \co* $, then $(\xi) =(\xi)_a$.
\end{proposition}
\begin{proof}
Assume otherwise. Then there is an increasing sequence of integers $n_k$, starting with $n_1=1$, such that $(\xi_a)_{n_k} \geq k^2\xi_{n_k}$ and for $k > 1$, 
$n_k$ can be chosen large enough so that for   $m_k:= \max\{j\in \mathbb{N} \mid\ \xi_j >k\xi_{n_k}\}$, $m_k \ge 2k n_{k-1}$, and also
$(\xi_a)_{\lceil \frac{m_k}{k} \rceil} \le\frac{1}{2k}\xi_{n_{k-1}}$.
First we show that $\sum_{i=1}^{n_k}\xi_i \leq 2\sum_{i=1}^{m_k}\xi_i$. Indeed,
\[
\sum_{i=m_k+1}^{n_k}\xi_i
\leq n_k\xi_{m_k+1}
\leq kn_k\xi_{n_k}
\leq \frac{n_k}{k}(\xi_a)_{n_k}
= \frac{1}{k}\sum_{i=1}^{n_k}\xi_i,
\]
thus
\[
\sum_{i=1}^{n_k}\xi_i
= \sum_{i=1}^{m_k}\xi_i + \sum_{i=m_k+1}^{n_k}\xi_i
\leq  \sum_{i=1}^{m_k}\xi_i + \frac{1}{k}\sum_{i=1}^{n_k}\xi_i,
\]
and hence
$\sum_{i=1}^{n_k}\xi_i
\leq \frac{1}{1-\frac{1}{k}}\sum_{i=1}^{m_k}\xi_i \leq 2\sum_{i=1}^{m_k}\xi_i$.
Now, define the sequence
\[
\eta_j=
\begin{cases}
\xi_{n_{k-1}} \quad&\text{for \quad $j\in[n_{k-1},
\lceil\frac{m_k}{k} \rceil)$}\\
\xi_{n_k}  &\text{for \quad $j\in [\lceil\frac{m_k}{k} \rceil, n_k]$}
\end{cases}.
\]
Clearly,  $\eta\in\co*$  and if $j\in [\lceil\frac{m_k}{k} \rceil, n_k],$ then
\begin{align*}
j(\eta_a)_j
&=\sum_{i=1}^{j}\eta_i
\geq \sum_{i=n_{{k-1}+1}}^{\lceil\frac{m_k}{k} \rceil}\eta_i
= (\lceil\frac{m_k}{k} \rceil - n_{k-1})\xi_{n_{k-1}}
\geq \frac{1}{2}\lceil\frac{m_k}{k} \rceil 2k (\xi_a)_{\lceil
\frac{m_k}{k}\rceil}\\
&\geq m_k (\xi_a)_{\lceil \frac{m_k}{k} \rceil}
\geq m_k(\xi_a)_{m_k}
\geq \frac{1}{2}\sum_{1}^{n_k}\xi_i
\geq  \frac{1}{2}\sum_{1}^{j}\xi_i
= \frac{1}{2}j(\xi_a)_j.
\end{align*}
By using this inequality, if
$j\in[n_{k-1}, \lceil\frac{m_k}{k} \rceil]$,
then also
\[
j(\eta_a)_j
= \sum_{i=1}^{n_{k-1}}\eta_i+\sum_{i=n_{k-1}+1}^{j}\eta_i
\geq \frac{1}{2}\sum_{i=1}^{n_{k-1}}\xi_i+\sum_{i=n_{k-1}+1}^{j}\xi_i
\geq \frac{1}{2}\sum_{i=1}^{j}\xi_i
= \frac{1}{2}j(\xi_a)_j.
\]
This proves that $(\eta)_a\supset (\xi)_a$ and hence $(\eta)\supset\widehat{(\xi)}$.
On the other hand,
\[
\xi_{m_k} >    k \xi_{n_k}
= k\eta_{ \lceil\frac{m_k}{k} \rceil}
= k(D_k\eta)_{m_k}
\]
and hence $(\eta) \not\supset (\xi)$, that is, $(\xi) \ne
\widehat{(\xi)}$, against the hypothesis.
\end{proof}

It remains to show that there exist principal ideals $ I=\widehat{I}$
with $I\neq F$.

\begin{example}\label{E:7} Let $\xi_j = \frac{1}{k}$ for $j\in
((k-1)!, k!]$. Then $(\xi) = \widehat{(\xi)}$.
\end{example}
\begin{proof}
It is enough to show that if  $\eta \in \co*$ and $(\xi) \not\subset (\eta)$, then $(\xi)_a \not\subset (\eta)_a$.
Assume without loss of generality that $\eta_1 = 1$. 
Since $\xi \not\in \Sigma((\eta))$, choose an increasing sequence $r_j\in \mathbb{N}$ for which $\xi_{r_j} \geq j\eta_{\lceil\frac{r_j}{j}\rceil}$.
Let $r_j\in ((k_j-1)!, k_j!]$ and assume without loss of generality that $k_j$ is strictly increasing. 
Then $\frac{1}{k_j}=\xi_{k_j!} = \xi_{r_j}\geq j\eta_{\lceil\frac{r_j}{j}\rceil} \geq j\eta_{\frac{k_j!}{j}}$ and hence
\[
(k_j+1)!(\eta_a)_{(k_j+1)!}
\leq \sum_{i=1}^{\frac{k_j!}{j}}\eta_1 + \sum_{i=\frac{k_j!}{j}+1}^{(k_j+1)!}\eta_{\frac{k_j!}{j}}
\leq \frac{k_j!}{j}+(k_j+1)!\eta_{\frac{k_j!}{j}}
\le (1+\frac{k_j+1}{k_j})\frac{k_j!}{j}
\le   3\frac{k_j!}{j}.
\]

On the other hand,
\[
(k_j+1)!(\xi_a)_{(k_j+1)!}
\geq \sum_{i=k_j!+1}^{(k_j+1)!}\frac{1}{k_j+1}
=\frac{k_j}{k_j+1}k_j!
\ge \frac{1}{2} k_j!.
\]
Combining the two inequalities, $(\xi_a)_{(k_j+1)!} \ge  \frac{j}{6}(\eta_a)_{(k_j+1)!}$, whence $\xi_a \neq O(\eta_a)$ and hence $(\xi_a) \not\subset (\eta_a)$.
\end{proof}


\begin{thebibliography}{99}


\bibitem{Aljan-Aran77}
Aljan{\v{c}}i\'c ,  S. and  Arandelovi\'c, D. 
     $O$-regularly varying functions,
     Publ. Inst. Math. (Beograd) (N.S.)~
     {\bf 22(36)}~(1977), 5--22.

\bibitem{gAlS78}
Allen, G. D. and Shen, L. C., \textit{On the structure of principal ideals of operators,} Trans. AMS \textbf{238} (1978), pp.~253--270.

\bibitem{BGT89}
Bingham, N. H. , Goldie, C. M., and Teugels, J. L.,
\textit{Regular Variation,} Encyclopedia of Mathematics and its
Applications (27), Cambridge University Press, 1989.  

\bibitem{aB89}
Blass, A., \textit{Applications of superperfect forcing and its
relatives,} Set Theory and its Applications, Lecture Notes in
Mathematics 1401, Springer-Verlag (1989), 18-40.

\bibitem{aB90}
Blass, A., \textit{Groupwise density and related cardinals,} Arch.
Math. Logic 30 (1990) 1-11.


\bibitem{aBgW78}
Blass, A. and Weiss, G.,  \textit{A Characterization and Sum
Decomposition of Operator Ideals,} Trans. Amer. Math. Soc.
\textbf{246} (1978), 407-417.

\bibitem{BPS71} Brown, A., Pearcy, C., and Salinas, N.,
\textit{Ideals of compact operators on Hilbert space,} Trans. Amer.
Math. Soc. \textbf{18} (1971), 373--384.

\bibitem{jC41}
Calkin, J. W., \textit{Two-sided ideals and congruences in the ring
of bounded operators in Hilbert space,}
Ann. of Math. (2) \textbf{42} (1941), pp.~839--873.

\bibitem{aC94}
Connes, A., \textit{Non Commutative Geometry,} San Diego Academic Press, 1994.

\bibitem{DFWW}
Dykema, K., Figiel, T., Weiss, G., and Wodzicki, M., \textit{The
commutator structure of operator ideals,}
Adv. Math., 185/1 pp. 1--79.

\bibitem{kDnK98}
Dykema, K. J. and Kalton, N. J., \textit{Spectral characterization of
sums of commutators. {I}{I},} J. Reine Angew. Math. \textbf{504}
(1998), pp.~127--137.

\bibitem{kDnK05}
Dykema, K. J. and Kalton, N. J., \textit{Sums of commutators in ideals and modules of type {II} factors}
Ann. Inst. Fourier (Grenoble) \textbf{55} 3 (2005), pp.~931--971.

\bibitem{kDgWmW00}
Dykema, K., Weiss, G. and Wodzicki, M., \textit{Unitarily invariant
trace extensions beyond the trace class,}
Complex analysis and related topics (Cuernavaca, 1996), Oper. Theory
Adv. Appl. \textbf{114} (2000), pp.~59--65.


\bibitem{vKgW02}
Kaftal, V. and Weiss, G., \textit{Traces, ideals, and arithmetic
means,} Proc. Natl. Acad. Sci. USA \textbf{99} (11) (2002),
pp.~7356--7360.

\bibitem{vKgW04-Traces}
Kaftal, V. and Weiss, G., \textit{Traces on operator ideals and
arithmetic means,} preprint.

\bibitem{vKgW04-Soft}
Kaftal, V. and Weiss, G., \textit{Soft ideals and arithmetic mean
ideals,} IEOT, to appear.

\bibitem{vKgW04-2nd order and cancellation}
Kaftal, V. and Weiss, G., \textit{2nd order arithmetic means and
cancellation for operator ideals,} Operators and Matrices (OAM), to appear.

\bibitem{vKgW07-OT21}
Kaftal, V. and Weiss, G., \textit{A survey on the interplay between arithmetic mean ideals, traces, lattices of operator ideals, and an infinite Schur-Horn majorization theorem}, Proceedings of the 21st International Conference on Operator Theory, Timisoara 2006 (Theta Bucharest), to appear

\bibitem{vKgW07-Majorization}
Kaftal, V. and Weiss, G., \textit{Majorization for infinite sequences, an extension of the Schur-Horn Theorem, and operator ideals}, in preparation.

\bibitem{nK87}
Kalton, N. J., \textit{Unusual traces on operator ideals,} Math.
Nachr. \textbf{134} (1987), pp.~119--130. 

\bibitem{nK89}
Kalton, N. J., \textit{Trace-class operators and commutators,} J.
Funct. Anal. \textbf{86} (1989), pp.~41--74.

\bibitem{nK98}
Kalton, N. J., \textit{Spectral characterization of sums of
commutators {I},} J. Reine Angew. Math. \textbf{504} (1998),
pp.~115--125.

\bibitem{hM01}
Mildenberger, H.,
\textit{Groupwise dense families,} Arch. Math. Logic \textbf{40}
(2001), no. 2, pp. 93--112.

\bibitem{nS74}
Salinas, N., \textit{Symmetric norm ideals and relative conjugate
ideals,} Trans. Amer. Math. Soc. \textbf{188} (1974), pp.~213--240.


\bibitem{jV89}
Varga, J., \textit{Traces on irregular ideals,} Proc. Amer. Math.
Soc. \textbf{107} 3 (1989), pp.~715--723.

\bibitem{gW75}
Weiss, G., \textit{Commutators and Operators Ideals,} dissertation
(1975), University of Michigan Microfilm.


\bibitem{gW80}
Weiss, G., \textit{Commutators of {H}ilbert-{S}chmidt operators
{I}{I},} IEOT \textbf{3} (4) (1980), pp.~574--600.

\bibitem{gW86}
Weiss, G., \textit{Commutators of {H}ilbert-{S}chmidt operators
{I},} IEOT \textbf{9} (6) (1986), pp.~877--892.

\bibitem{gW05}
Weiss, G., \textit{B(H)-commutators: a historical survey,} 
Recent advances in operator theory, operator algebras, and their applications, pp.~307--320, Oper. Theory Adv. Appl., 153, BirkhŠuser, Basel, 2005.

\bibitem{mW94}
Wodzicki, M., \textit{Algebraic {$K$}-theory and functional analysis,}
First European Congress of Mathematics, Vol.\ II (Paris,1992)
\textbf{120} (1994), pp.~485--496. 


\bibitem{mW02}
Wodzicki, M., \textit{Vestigia investiganda,} Mosc. Math. J.
\textbf{4} (2002), pp.~769--798, 806. 

\end{thebibliography}
\end{document}